\documentclass[reqno, 11pt]{amsart}
\usepackage{amsmath}
\usepackage{amsthm}
\usepackage{amsfonts}
\usepackage{amssymb}
\usepackage{mathrsfs}
\usepackage{epsfig}
\usepackage{slashed}
\usepackage{mathtools}
\usepackage{enumerate}

\newcommand{\mb}{\mathbf}
\newcommand{\mc}{\mathcal}

\newcommand{\rg}{\mathrm{rg}\,}
\newcommand{\N}{\mathbb{N}}
\newcommand{\R}{\mathbb{R}}
\newcommand{\C}{\mathbb{C}}

\newcommand{\B}{\mathbb{B}^d}

\newcommand{\Ceven}{C_{e}^{\infty}[0,1]}

\newcommand{\bs}[1]{\boldsymbol{#1}}
\newcommand{\D}{D_d}
\newcommand{\K}{K_d}
\newcommand{\del}{\nabla_{\text{rad}}}

\newcommand{\rst}[1]{\ensuremath{{\mathbin |}%
\raise-.5ex\hbox{$#1$}}} 

\newcommand{\norm}[1]{{\left\vert\kern-0.25ex\left\vert\kern-0.25ex\left\vert #1 
    \right\vert\kern-0.25ex\right\vert\kern-0.25ex\right\vert}}

\newtheorem{lemma}{Lemma}[section]
\newtheorem{theorem}[lemma]{Theorem}
\newtheorem{corollary}[lemma]{Corollary}
\newtheorem{proposition}[lemma]{Proposition}
\theoremstyle{remark}
\newtheorem{remark}{Remark}

\theoremstyle{definition}
\newtheorem{definition}[lemma]{Definition}

\numberwithin{equation}{section}

\title[Blowup for wave equations on $\R^{d+1}$]{Stable blowup for wave equations in odd space dimensions}

\author{Roland Donninger}
\address{Rheinische Friedrich-Wilhelms-Universit\"at Bonn, 
Mathematisches Institut,  Endenicher Allee 60, D-53115 Bonn, Germany}
\email{donninge@math.uni-bonn.de}

\author{Birgit Sch\"orkhuber}
\address{Universit\"at Wien, Fakult\"at f\"ur Mathematik,
Oskar-Morgenstern-Platz 1, A-1090 Vienna, Austria}
\email{birgit.schoerkhuber@univie.ac.at}

\thanks{Roland Donninger is supported by the Alexander von Humboldt Foundation via
a Sofja Kovalevskaja Award endowed by the German Federal Ministry of Education
and Research. Birgit Sch\"orkhuber is supported by the Austrian Science Fund
(FWF) via the Hertha Firnberg Program, Project Nr. T739-N25. Both authors would like to thank the Hausdorff
Research Institute for Mathematics in Bonn for the hospitality during the Trimester
Program 'Harmonic Analysis and Partial Differential Equations'. Partial support by the Deutsche Forschungsgemeinschaft (DFG), CRC 1060 'The Mathematics of Emergent Effects' is also gratefully acknowledged.}

\begin{document}
\begin{abstract}
We consider semilinear wave equations with focusing power nonlinearities in odd space dimensions $d \geq 5$. We prove that for every $p > \frac{d+3}{d-1}$  there exists an open set of radial initial data in $H^{\frac{d+1}{2}} \times H^{\frac{d-1}{2}}$ such that the corresponding solution exists in a backward lightcone and approaches the ODE blowup profile. The result covers  the entire range of energy supercritical nonlinearities and extends our previous work for the three--dimensional radial wave equation to higher space dimensions. 
\end{abstract}

\maketitle

\section{Introduction}

We consider the initial value problem for the focusing nonlinear wave equation
\begin{align}\label{Eq:NLW} 
\begin{split}
& \partial^2_{t} u- \Delta u = |u|^{p-1} u, \\
& u|_{t=0} = u_0, \quad  \partial_t u|_{t=0}= u_1,
\end{split}
\end{align}
for $(t,x) \in I \times \R^{d}$, $d = 2k+1$, $k \geq 2$ and $I$ an interval, where $0 \in I$.  Eq.~\eqref{Eq:NLW} is conformally invariant for $p = \frac{d+3}{d-1}$ and we restrict ourselves to the \textit{superconformal} case 
\begin{align}\label{Eq:superconf}
 p > \frac{d+3}{d-1}.
 \end{align}
The above equation enjoys scaling invariance in the sense that if $u$ solves Eq.~\eqref{Eq:NLW} then another solution can be obtained by setting $u_{\lambda}(t,x) := \lambda^{-\frac{2}{p-1}} u(t/\lambda, x/\lambda)$ for $\lambda > 0$. The conserved energy is given by
\[ E(u(t,\cdot), \partial_t u(t,\cdot)) = \tfrac{1}{2} \|(u(t,\cdot), \partial_t u(t,\cdot)) \|^2_{\dot H^1 \times L^2(\R^d)} - \tfrac{1}{p+1} \| u(t,\cdot) \|^{p+1}_{L^{p+1}(\R^d)} \]
and it is invariant under the above scaling for $p = \frac{d+2}{d-2}$, which defines the \textit{energy critical} case. In general, the scaling invariant Sobolev spaces are $\dot H^{s_p}\times \dot H^{s_p-1}(\R^d)$, where the index $s_p = \frac{d}{2} - \frac{2}{p-1}$ is usually referred to as the critical regularity.

\subsection{Basic well-posedness theory and explicit blowup solutions} 
One is usually interested in (strong) solutions of Eq.~\eqref{Eq:NLW} that satisfy the equation in integral form by using Duhamel's principle, see for example \cite{Tao}. In this sense, Eq.~\eqref{Eq:NLW} is locally well-posed in $\dot H^{s_p} \times \dot H^{s_p-1}(\R^d)$ for $d \geq 5$ and $p > \frac{d+3}{d-1},$ given that the nonlinearity is sufficiently regular, cf.~ Lindblad and Sogge \cite{LindbladSogge95}. Moreover, solutions that correspond to sufficiently small initial data can be extended globally in time. We also note that local well-posedness in $\dot H^{s} \times \dot H^{s-1}(\R^d)$ for $s > \frac{d}{2}$ and smooth nonlinearities is classical \cite{Tao}. However, global well-posedness does not hold in general. A convexity argument by Levine \cite{Lev74} shows that initial data with negative energy (and finite $L^2-$norm) lead to blowup in finite time, cf.~ also \cite{KilStoVis12} for generalizations.

Explicit examples for singularity formation can be obtained by considering the so called ODE blowup solution 
\begin{align}\label{Eq:BlowUpSol}
 	 u_T(t,x)= c_p (T-t)^{-\frac{2}{p-1}},  \quad   c_p := \left [\tfrac{2(p+1)}{(p-1)^2} \right]^{\frac{1}{p-1}},
\end{align}
which is independent of the space dimension and solves the ordinary differential equation $u_{tt} = |u|^{p-1} u$ for $p >1$. By finite speed of propagation one can use $u_T$ to construct compactly supported smooth initial data such that the solution blows up as $t \to T$. 

In one space dimension the ODE blowup mechanism is universal, cf.~the fundamental work by Merle and Zaag \cite{MerZaa07}, \cite{MerZaa08}, \cite{MerZaa12a}, \cite{MerZaa12b} and the references therein. In higher dimensions, the situation is more complex. Depending on $d$ and $p$ many other explicit examples for singular solutions were found in the past years, including the celebrated work of Krieger, Schlag and Tataru \cite{KriSchTat09} on \textit{type II} blowup solutions for the energy critical equation in three space dimensions, see below. For $d=3$, $p=3$ and $p\geq 7$ an odd integer, it was proved by Bizo\'n, Breitenlohner, Maison and Wasserman \cite{BBMW10}, \cite{BMW07} that Eq.~\eqref{Eq:NLW} admits infinitely many radial \textit{self-similar} blowup solutions of the form $(T-t)^{-\frac{2}{p-1}} f_n(\frac{|x|}{T-t})$, $n \in \N_0$, with $u_T$ corresponding to the groundstate, i.e., $f_0 = c_p$. Another blowup mechanism for Eq.~\eqref{Eq:NLW}, which only exists for $d \geq 11$ and a range of supercritical nonlinearities $p > p(d) > \frac{d+2}{d-2}$, was recently established by Collot \cite{Collot2014}, see below.

Most of these explicit solutions have unstable directions, i.e., they are unstable under generic small perturbations and are not supposed to describe the 'typical' blowup behavior for solutions of Eq.~\eqref{Eq:NLW}, see for example \cite{KriNahas12}. On the other hand, numerical experiments by Bizo\'n, Chmaj and Tabor \cite{BCT04} for the three--dimensional equation show that the behavior of generic radial blowup solutions can be characterized in terms of the ODE blowup solution locally around the blowup point.  

The stability of $u_T$ in three space dimensions was established in our previous works \cite{DonSch12}, \cite{DonSch13} for radial perturbations  and all $p >1$. Recently, we could extend this to the general case (without symmetry) \cite{DonSch14b} for $p>3$. For subconformal nonlinearities the dynamics around $u_T$ were also investigated by Merle and Zaag \cite{MerZaa13a} in the non-radial setting and in arbitrary space dimensions.

In view of the findings in \cite{Collot2014} for supercritical radial wave equations in high space dimensions, we extend our previous results and establish the stability of the ODE blowup solution in arbitrary odd space dimensions. Although for $d=3$ we were able to drop the symmetry assumption, it is an open question how this can be accomplished for $d \geq 5$, see the discussion below. We therefore restrict ourselves to the radial case and study solutions that blow up at the origin (which is the most interesting case). 

We note that this work is not a mere technical generalization of \cite{DonSch12}, \cite{DonSch13}. It can rather be viewed as a systematization and refinement of our approach that has also been applied (with slight modifications) to establish stable self-similar blowup for equivariant wave maps \cite{DonSchAic12}, \cite{DonSch12} and Yang-Mills fields \cite{Don11} in supercritical dimensions.

\subsection{Radial solutions in lightcones}

In the following we use the abbreviation $u[t]:=(u(t,\cdot),\partial_t u(t,\cdot))$. We are interested in the behavior of radial solutions of Eq.~\eqref{Eq:NLW} in backward lightcones  
\[ \Gamma_{T}(\R^d) := \{ (t,x) \in [0,T) \times \R^d:  |x| \leq T - t  \},\]
with vertex $(T,0)$ for $T>0$. Consequently, a suitable concept of (strong) solutions in lightcones is required. This can be obtained for example by combining the classical Duhamel formula on $\R^d$ with suitable cut-off techniques, see \cite{KilStoVis12}. Here, we pursue another approach which is based on the formulation of Eq.~\eqref{Eq:NLW} in self-similar coordinates 
\[\tau := -\log(T-t) + \log T, \quad \xi := \frac{x}{T-t}.\]
To motivate the following, let  $u \in C^{\infty}(\Gamma_{T})$ be a radial solution of Eq.~\eqref{Eq:NLW}. By setting
\begin{align}\label{Def:psi}
\begin{split}
\psi_1^{T}( -\log(T-t) + \log T,\tfrac{x}{T-t})&:=(T-t)^{\frac{2}{p-1}}u(t,x),\\
\psi_2^{T}( -\log(T-t) + \log T,\tfrac{x}{T-t})&:=(T-t)^{\frac{2}{p-1}+1} \partial_t u(t,x),
\end{split}
\end{align} 
we obtain a smooth solution of the first order system
\begin{align}
\label{Eq:CssSystem}
\begin{split}
\partial_{\tau} \psi_1^{T}&=\psi_2^{T}  -\xi \cdot \nabla \psi_1^{T}
-\tfrac{2}{p-1}\psi_1^{T}, \\
\partial_{\tau}  \psi_2^{T}&=\Delta \psi_1^{T}
-\xi \cdot \nabla  \psi_2^{T}-\tfrac{p+1}{p-1}\psi_2^{T}
+\psi_1^{T} |\psi_1^{T}|^{p-1}.
\end{split}
\end{align}
By setting $\Psi^{T} (\tau) := (\psi^{T}_1(\tau,\cdot) ,\psi^{T}_2(\tau,\cdot) )$ this can be written as
\begin{align*}
\partial_{\tau} \Psi^{T} (\tau) = \mc L_0 \Psi^T(\tau) + \mb F(\Psi^T(\tau)),
\end{align*}
where $\mc L_0 $ represents the linear part of the right hand side of Eq.~\eqref{Eq:CssSystem}. 
To formulate the following statement we define for $k \in \N_0$
 \[ H_{\mathrm{rad}}^k(\B) := \{ u \in H^k(\B): u \text{ is radial} \}.  \]

\begin{proposition}\label{Prop1}
Let  $\mathfrak H := H_{\mathrm{rad}}^{\frac{d+1}{2}} \times H_{\mathrm{rad}}^{\frac{d-1}{2}}(\B)$. There is a dense domain $\mc D(\mc L_0) \subset \mathfrak H$ such that the operator $\mc L_0: \mc D(\mc L_0) \subset \mathfrak H \to \mathfrak H$ is closable and its closure generates a strongly continuous one parameter semigroup $\{\mc S_0(\tau): \tau \geq 0 \}$ of bounded operators on $\mathfrak H$.
\end{proposition}

This is an immediate consequence of the results proved in Section \ref{Sec:LinWellPosed}.
By Duhamel's principle we can now formulate the above equation as an abstract integral equation
\begin{align}\label{Eq:IntegralIntro}
\Psi^{T} (\tau) = \mc  S_0(\tau) \Psi^T(0) + \int_0^{\tau}  \mc S_0(\tau - \tau') \mb F(\Psi^T(\tau')) d\tau'.
\end{align}

We take this as a defining equation for our notion of strong lightcone solutions.

\begin{definition}
We say that $u$: $\Gamma_T(\R^d)  \to \R$ is a radial $(H^{\frac{d+1}{2}}$--)solution of Eq.~\eqref{Eq:NLW} if the corresponding $\Psi^T$ belongs to $C([0,\infty), \mathfrak H)$ and satisfies Eq.~\eqref{Eq:IntegralIntro} for all $\tau \geq 0$.
\end{definition}

\begin{definition}
Let $(u_0,u_1) \in  H_{\mathrm{rad}}^{\frac{d+1}{2}} \times H_{\mathrm{rad}}^{\frac{d-1}{2}}(\R^d)$. We say that $T > 0$ belongs to $\mc T(u_0,u_1) \subseteq (0,\infty)$ if there exists a solution $u$: $\Gamma_T(\R^d)  \to \R$ to Eq.~\eqref{Eq:NLW} with $u[0]=(u_0,u_1)|_{\B_T}$. Set
\[ T_{(u_0,u_1) } := \sup \{\mc T(u_0,u_1)  \cup \{0\} \}. \]
If $T_{(u_0,u_1) } < \infty$, we call $T_{(u_0,u_1) }$ the \textit{blowup time at the origin}.
\end{definition}

\subsection{The main result}\label{Sec:MainResult}  We prove the stability of the ODE blowup solution in the following sense.
\begin{theorem}\label{Th:Main} For $d= 2k+1$, $k \in \N$, $k \geq 2$, fix $p > \frac{d+3}{d-1}$ and $T_0 > 0$. There are constants $M, \delta > 0$ such that if $u_0,u_1$ are radial functions with
\begin{align}\label{Eq:CondData}
\|(u_0,u_1)-  u_{T_0}[0] \|_{H^{\frac{d+1}{2}} \times H^{\frac{d-1}{2}}(\B_{T_0+\delta})}  < \tfrac{\delta}{M}
\end{align}
the following statements hold:
\begin{enumerate}[i)]
\item The blowup time at the origin $T :=T_{(u_0,u_1) }$ is contained in the interval $[T_0 - \delta, T_0 + \delta]$. 
\item The solution 
$u: \Gamma_T(\R^d)  \to \R$ satisfies  
\begin{align}\label{Eq:DecayRates_1}
\begin{split}
(T-t)^{\frac{2}{p-1}} \| u(t,\cdot) - u_T(t,\cdot) \|_{L^{\infty}(\B_{T-t})}  \lesssim (T-t)^{\mu_p}, \\
(T- t)^{-s_p} \|u(t,\cdot) - u_T(t,\cdot) \|_{L^2(\B_{T-t}) }  \lesssim (T-t)^{\mu_p}, \\
(T- t)^{-s_p + 1}  \|u[t] - u_T[t] \|_{\dot H^1 \times  L^{2}(\B_{T-t}) }  \lesssim (T-t)^{\mu_p}, \\
\end{split}
\end{align}
for $s_p = \frac{d}{2} - \frac{2}{p-1}$, $\mu_p = \mathrm{min} \{\frac{2}{p-1},1 \} - \varepsilon$ and $\varepsilon >0$ small. Furthermore, for $j = 2, \dots, \frac{d+1}{2}$,
\begin{align}\label{Eq:DecayRates_2}
\begin{split}
(T- t)^{-s_p + j}  \|u[t] \|_{\dot H^j \times \dot H^{j-1}(\B_{T-t}) }  \lesssim  (T-t)^{\mu_p}. \\
\end{split}
\end{align}
\end{enumerate}
\end{theorem}

\begin{remark}
The normalizing factors in Eq.~\eqref{Eq:DecayRates_1} and Eq.~\eqref{Eq:DecayRates_2} appear naturally and reflect the behavior of $u_T$ in the respective norms. Since the ODE blowup solution has a trivial spatial profile, it vanishes identically in higher order homogeneous Sobolev norms, which yields Eq.~\eqref{Eq:DecayRates_2}. The $\varepsilon$--loss in the convergence rates is due to the application of abstract arguments from semigroup theory. 
\end{remark}

\begin{remark}
Our approach is perturbative, i.e., we construct solutions of the form $u = u_T + \varphi$.  This and the embedding $H^{\frac{d+1}{2}} \hookrightarrow L^{\infty}$ guarantee that the nonlinearity is smooth for all $p > 1$ provided that the perturbation is small enough. In particular, Theorem \ref{Th:Main} can be extended to all $p >1$ without modifications. We are therefore able to construct solutions of Eq.~\eqref{Eq:NLW} (at least in a backward lightcone) for nonlinearities that are not covered by the standard local well-posedness theory.
\end{remark}

If we restrict ourselves to Sobolev spaces of integer order, then the regularity required in Theorem \ref{Th:Main} is optimal for $p > 5$ by local well-posedness ($\lceil s_p \rceil = \frac{d+1}{2} $). However, for $p \leq 5$ this might be improved and we show this explicitly for $p=3$ in Theorem \ref{Th:Main2} below. In this case, ${\lceil s_3 \rceil} =\frac{d -1}{2}$ and $H^{\lceil s_3 \rceil}$--solutions can be defined in a similar manner as above.

\begin{theorem}[Improvement of the toplogy]\label{Th:Main2}
Let $p=3$. Then the first statement of Theorem \ref{Th:Main} holds (mutatis mutandis) for radial initial data satisfying Eq.~\eqref{Eq:CondData} in $H^{\lceil s_p \rceil} \times H^{\lceil s_p \rceil -1}$, where $s_p = \frac{d}{2} - 1$. The corresponding solution $u: \Gamma_T(\R^d) \to \R$ satisfies the estimates
\begin{align*}
\begin{split}
(T- t)^{-s_p} \|u(t,\cdot) - u_T(t,\cdot) \|_{L^2(\B_{T-t}) }  \lesssim (T-t)^{\frac12-\varepsilon}, \\
(T- t)^{-s_p + j}  \|u[t] - u_T[t] \|_{\dot H^j \times \dot H^{j-1}(\B_{T-t}) }  \lesssim (T-t)^{\frac12-\varepsilon}, \\
\end{split}
\end{align*}
for $j = 1, \dots, \lceil s_p \rceil$.
\end{theorem}

We note that there are possibly other situations where the topology can be optimized and it will become clear in Section \ref{Sec:Outline} how this could be realized within our framework. However, we do not pursue this here.

\subsection{Related results}

Blowup for the wave equation with (sub)conformal focusing nonlinearities  $1 < p \leq \frac{d+3}{d-1}$, $d \geq 2$, was considered in the seminal work of Merle and Zaag \cite{MerZaa03}, \cite{MerZaa05}, cf. also Antonini and Merle \cite{AntMer01}. They were able to prove that \textit{all} blowup solutions diverge with the self-similar rate in the backward lightcone of the blowup point (at the energy level). They also extended their previous analysis for the one-dimensional wave equation to $d$ dimensions in the subconformal case to characterize the behavior of radial solutions provided that the blowup occurs outside the origin \cite{MerZaa11}. In \cite{MerZaa13a}, \cite{MerZaa13b} they studied the dynamics around $u_T$ (without symmetry assumptions) and investigated properties of the blowup surface.

In the superconformal regime, much less is known concerning the behavior of generic solutions. However, for energy subcritical nonlinearities, Killip, Stovall and Visan \cite{KilStoVis12} as well as Hamza and Zaag \cite{HamZag13} were able to derive upper bounds for the blowup rate.

In the energy critical case, $u_T$ is the unique self-similar solution and it can be used to construct blowup solutions that diverge in the scale invariant norm $\dot H^1 \times L^2(\R^3)$, cf.~ for example \cite{KriWong13}. This behavior is referred to as \textit{type I} and contrasted by \textit{type II} blowup, where solutions stay bounded in the critical norm. First examples of type II solutions were obtained by Krieger, Schlag and Tataru \cite{KriSchTat09}, \cite{KriSch12} for the radial equation in three dimensions using the (up to scaling) unique solution of the corresponding elliptic problem. We also refer to \cite{DonKrieger13} for solutions that blow up in infinite time.  A detailed description of all possible type II blowup dynamics was provided in the celebrated work of Duyckaerts, Kenig and Merle, cf.~\cite{DuyKenMer12c} as well as \cite{DuyKenMer14} for further references. For $d=4$, smooth type II blowup solutions were constructed by Hillairet and Rapha\"el \cite{HillRaph12}, and we also refer to \cite{Jen2015} for a recent result by Jendrej in five dimensions.

To the best knowledge of the authors, all currently available results for supercritical nonlinearities $p > \frac{d+2}{d-2}$ are either conditional or consider perturbations around certain special solutions. Type II blowup behavior for radial solutions was excluded by Duyckaerts, Kenig and Merle \cite{DuyKenMerle12} for $d=3$ and by Dodson and Lawrie \cite{DodsonLawrie14} for $d=5$. We also refer to similar results for the defocusing case obtained  in \cite{KenMer11}, \cite{KilVis11} or \cite{Bul11}. These works show that if the critical norm stays bounded up to the maximal time of existence, then the solution is global in time. We also mention a recent work by Krieger and Schlag \cite{KriSch14}, where smooth global solutions are constructed that have infinite critical norm (but are bounded in all higher norms). Recently, based on the pioneering work of Merle, Rapha\"{e}l and Rodnianski \cite{MerRaphRod2014} for the energy supercritical nonlinear Schr\"odinger equation, a new blowup mechanism was described by Collot \cite{Collot2014} for the radial wave equation in $d \geq 11$ and $p >  1 +  \frac{4}{d-4-2 \sqrt{d-1}}$. There, solutions blow up via concentration of the soliton profile, which is somewhat reminiscent of the type II behavior in the energy critical case. However, the solutions diverge in the critical norm and this blowup mechanism could therefore be referred to as type IIb in order to avoid confusion.

\subsection{Strategy of the proof}\label{Sec:Outline}

 We consider the radial equation 
\begin{align}\label{Eq:RadialWave}
\partial^2_t u  -   \tfrac{1}{r^{d-1}}\partial_r ( r^{d-1}  \partial_r u) = |u|^{p-1} u,
\end{align}
for $u = u(t,r)$ and initial data $(u_0,u_1) = u_{T_0}[0] + (f,g)$, where $(f,g)$ are free radial functions. We study the initial value problem in a backward lightcone  $\{(t,r):  t \in [0,T) , r \in [0,T-t]\}$, where $T > 0$ is a parameter that will be fixed in the final step of the proof. We introduce rescaled variables and rewrite Eq.~\eqref{Eq:RadialWave} as a first order system in (radial) similarity coordinates 
\[\tau := -\log(T-t) + \log T, \quad \rho := \tfrac{r}{T-t},\]
which yields the abstract evolution problem 
\begin{align*}
\partial_{\tau} \Psi = \mb L_0 \Psi + \mb F(\Psi),
\end{align*}
where $\Psi = (\psi_1,\psi_2)$. Here, $\mb L_0$ represents the radial wave operator and $\mb F(\Psi) = (0, |\psi_1|^{p-1}\psi_1)$. The backward lightcone now corresponds to \[\{ (\tau, \rho): \tau \in [0,\infty), \rho \in [0,1] \}.\] 
Note that the parameter $T$ does not appear in the equation itself, but it shows up in the initial data. In this formulation, the ODE blowup solution corresponds to the static solution $\mb c_p$. The ansatz $\Psi = \mb c_p + \Phi$ yields
 \begin{align}\label{Eq:ODEPert}
\partial_{\tau} \Phi = \mb L \Phi + \mb N(\Phi),
\end{align}
with $\mb L := \mb L_0 + \mb L' $ representing the linearized part of the equation and $ \mb N$ denoting the nonlinear remainder.
Eq.~\eqref{Eq:ODEPert} is investigated as an abstract ODE on a Hilbert space with norm
\begin{align*}
\| \mb u \|^2 := \| u_1(|\cdot|) \|^2_{H^k(\B)} +  \| u_2(|\cdot|) \|^2_{H^{k-1}(\B)},
\end{align*}
for $\mb u = (u_1,u_2)$. For Theorem \ref{Th:Main}, we choose $k= (d+1)/2$, whereas for the proof of Theorem \ref{Th:Main2},
$k= (d-1)/2$ is sufficient. The choice of the topology is motivated as follows.

First, we have to derive suitable Lipschitz estimates for $\mb N$. In the first situation (Theorem \ref{Th:Main}), we can exploit the Sobolev embedding $H^{k} \hookrightarrow L^{\infty}$ for $k > \frac{d}{2}$ to infer that the nonlinear remainder is smooth for  all $p > 1$ (given that the perturbations are sufficiently small). An application of Moser's inequality then yields the desired result. For $p=3$ (Theorem \ref{Th:Main2}), the nonlinearity is analytic, hence regularity is not an issue and the Lipschitz estimates can be obtained by using H\"older's inequality and Sobolev embedding.

Furthermore, we need a decay estimate for the time evolution of the \textit{linear} wave equation in similarity coordinates. To see what can be expected, let us drop the symmetry assumption for a moment and let $u(t,x)$, $x \in \R^d$, be a generic solution of the free wave equation
\[ \partial^2_{t} u - \Delta u = 0.\]
Let $\psi_1$ denote the rescaled solution in similarity coordinates as in Eq.~\eqref{Def:psi}. The scaling behavior of  Sobolev norms implies that 
\begin{align*}
\| \psi_1(\tau, \cdot ) \|_{\dot H^k(\B)}  = (T-t)^{\frac{2}{p-1} - \frac{d}{2} + k} \|u(t,\cdot)\|_{\dot H^k(\B_{T-t})},
\end{align*}
where we write $\tau =- \log(T-t)   + \log T$ for brevity.
One can easily check that $\|u(t,\cdot)\|_{\dot H^k(\B_{T-t})}$ is bounded. However, without further assumptions on the regularity this cannot be improved since one can construct explicit solutions that decay arbitrarily slow in $\dot H^k(\B_{T-t})$. Hence, a decay estimate for $\psi_1$ in $\dot H^k(\B)$ can only be obtained if $\frac{2}{p-1} - \frac{d}{2} + k > 0$. Unfortunately, in backward lightcones such homogeneous quantities are only seminorms and we have to work instead with
\[  \| \psi_1( \tau , \cdot ) \|_{H^{k}(\B)}  = \sum_{j=0}^{k} (T-t)^{\frac{2}{p-1} - \frac{d}{2} + j}  \|u(t,\cdot)\|_{\dot H^j(\B_{T-t})}.\]
At first glance, the lower order terms seem to spoil the decay estimate. However, this can be overcome by considering equivalent norms. In the radial case, the construction is based on the reduction of the $d-$dimensional radial wave equation to the one--dimensional case (or simply to a lower dimensional equation, depending on the required level of regularity). For our purpose we set $Du(t,r) := (\frac{1}{r} \partial_r)^{\frac{d-3}{2}} (r^{d-2} u(t,r))$ and observe that
\begin{align*}
\sum_{j=0}^{(d-1)/2}  (T-t)^{\frac{2}{p-1} - \frac{d}{2} + j} &  \|u(t,|\cdot|)\|_{\dot H^j(\B_{T-t})} \\
&   \simeq (T-t)^{\frac{2}{p-1} - \frac12} \|D u(t,\cdot)\|_{\dot H^1(0,T-t)}.
\end{align*}
Since $Du$ solves the one-dimensional equation, $\|D u(t,\cdot)\|_{\dot H^1(0,T-t)}$ bounded. This equivalence is crucial in the proof of Theorem \ref{Th:Main2}. It is also obvious that we have decay only if $p<5$. For Theorem \ref{Th:Main}, this is not sufficient (also because the above quantity only provides $\frac{d-1}{2}$ derivatives). We could work instead with $\| Du(t,\cdot)\|_{\dot H^2(0,T-t)}$, but this is only a seminorm (the radial derivative of $Du$ does not vanish at the origin). Adding the energy part solves this problem, but spoils again the decay estimate for $p \geq 5$. Hence, we consider
\begin{align*}
\big \| (\psi_1(\tau,\cdot),\psi_2(\tau, \cdot))\|^2_{D} & := 
 (T-t)^{\frac{4}{p-1} +1} \|D u(t,\cdot) \|^2_{\dot H^2(0,T-t)}  \\
 &  +   (T-t)^{\frac{4}{p-1} +1} \| Du_t(t, \cdot)  \|^2_{\dot H^1(0,T-t)}  \\
  & +  (T-t)^{\frac{4}{p-1}} \big | (D u)_r(t,T-t) + (D u)_t(t,T-t) \big|^2
\end{align*}
and show that  $ \| \cdot \|_{H^{\frac{d+1}{2} }\times H^{\frac{d-1}{2}}(\B)}  \simeq  \| \cdot \|_{D} $.  Observe that solutions of the one--dimensional equation $w_{tt} - w_{rr} = 0$ satisfy
\begin{align*}
\tfrac{d}{dt} |w_r(t,T-t) + w_t(t,T-t)|^2 = 0.
\end{align*}
In view of this, the desired decay follows for all $p > 1$. We note that for $d=3$ such an equivalent norm also exists in the non-radial context, cf.~\cite{DonSch14b}. However, it is not clear how this can be generalized to arbitrary space dimensions.

To prove the results of Section \ref{Sec:MainResult} we proceed as in \cite{DonSch12}, \cite{DonSch13} and use the theory of strongly continuous one--parameter semigroups to address the linearized equation. Since the operator $\mb L$ is a highly non-selfadjoint object, semigroup theory can deploy its full strength and enables us to treat the problem on a very abstract level. 
\begin{itemize}
\item With the above considerations and suitable equivalent norms it is easy to show that $\mb L_0$ is the generator of a semigroup which satisfies a suitable decay estimate. Since $\mb L'$ is bounded, well-posedness of the linearized problem and the existence of a strongly continuous semigroup $(\mb S(\tau))_{\tau \geq 0}$ generated by $\mb L$ follow immediately. 

\item To deduce suitable growth estimates for the semigroup we analyze the spectrum of the generator. Compactness of the perturbation reduces matters to the investigation of the eigenvalue problem, which can be solved explicitly in terms of hypergeometric functions. We show that the spectrum of $\mb L$ is contained in a left half plane except for the point $\lambda=1$, which is an eigenvalue with eigenfunction $\mb g$. The existence of this unstable eigenvalue is a consequence of the time translation symmetry of the problem and we define a spectral projection $\mb P$ to analyze the behavior of solutions on the stable subspace. Note that we have to verify that $\mathrm{rg} \mb P = \langle \mb g \rangle$, since we are dealing with a non-selfadjoint problem. In contrast to \cite{DonSch12}, \cite{DonSch13}, where we used resolvent estimates and the Gearhardt-Pr\"uss Theorem to deduce growth bounds for $(1 - \mb P) \mb S(\tau)$, we employ a much simpler argument here and exploit the compactness of the perturbation directly. This is a substantial simplification relying only on standard results from semigroup theory. As a result we obtain that 
\begin{align*}
\|(1 - \mb P) \mb S(\tau) \| \lesssim e^{-\mu_p \tau}, \quad \mb P  \mb S(\tau) \mb u = e^{\tau} \mb u, \quad \text{where } \mu_p > 0.
\end{align*}

\item  We rewrite Eq.~\eqref{Eq:ODEPert} in Duhamel form,
\begin{align}\label{Eq:IntegralEq_Introduc}
\Phi(\tau) = \mb S(\tau) \mb U(\mb v, T)  + \int_0^{\tau} \mb S(\tau - \tau') \mb N(\Phi(\tau'))d\tau',
\end{align}
where $\mb U(\mb v,T)$ with $\mb v = (f,g)$ gives the original initial data. The main ingredients for the nonlinear theory are the above estimates for the semigroup and Lipschitz estimates for the nonlinearity of the form
\[ \|  \mb N(\mb u) - \mb N(\mb v) \| \lesssim  ( \| \mb u \| + \| \mb v \|) \| \mb u - \mb v\|. \]
The rest of the proof is purely abstract.
\item  We add a correction term to Eq.~\eqref{Eq:IntegralEq_Introduc} in order to suppress the unstable behavior of $\mb S(\tau)$. An application of the Banach fixed point theorem shows the existence of a unique solution to the modified equation 
\begin{align*}
\Phi(\tau) = \mb S(\tau) &  \mb U(\mb v, T) + \int_0^{\tau} \mb S(\tau - \tau') \mb N(\Phi(\tau'))d\tau'   \\
& - e^{\tau} \mb C(\Phi, \mb U(\mb v, T)),
\end{align*}
given that $\mb v$ is small and $T$ is close to $T_0$. Furthermore, the solution decays to zero with the linear decay rate. 

\item In the final step, we show that for every small $\mb v$ there exists a $T_{\mb v}$ close to $T_0$ such that $\mb C(\Phi, \mb U(\mb v, T_{\mb v}))= 0$. We exploit the fact that  $ \mb C(\Phi, \mb u) \in \langle \mb g \rangle$ and apply the Brouwer fixed point theorem as in \cite{DonSch14b}. This is a substantial simplification compared to \cite{DonSch12}, \cite{DonSch13}, where differentiability of several quantities was required. Transforming back to original coordinates yields the result.
\end{itemize}

\section{Notation}	
Throughout the paper we assume that $d = 2k +1$, $k \in \N$, $k \geq 2$ is fixed. We write $\N$ for the natural numbers $\{1,2,3, \dots\}$ and  set $\N_0 := \{0\} \cup \N$. As mentioned above, $\B_R$ denotes the open ball in $\R^d$ centered at zero with radius $R > 0$. If $R=1$, we simply write $\B$. The notation $a\lesssim b$ means $a\leq Cb$ for an absolute constant $C>0$ and we also write $a\simeq b$ if $a\lesssim b$ and $b \lesssim a$.

For a function $x \mapsto g(x)$, we use the notation $g^{(n)}(x) = \frac{d^n g(x)}{dx^n} $ for derivatives of order $n \in \N$. For $n=1,2$ we also write $g'(x)$ and $g''(x)$, respectively. For a function $(x,y) \mapsto f(x,y)$ partial derivatives of order $n$ will be denoted by $\partial^{n}_x f (x,y) = \frac{\partial^n}{\partial x^n} f(x,y) = \partial^{n}_1 f (x,y) $.  For $\Omega \subset \R^d$ a domain, $H^{m}(\Omega)$, $m \in \N_0$, denotes the standard Sobolev space with norm
\[ \|  u \|^2_{H^m(\Omega)} := \sum_{\alpha: |\alpha| \leq m} \|\partial^{\alpha}  u \|^2_{L^2(\Omega)},\]
where $\alpha = ( \alpha_1, \dots, \alpha_d) \in \N_0^{d}$ is a multi-index, i.e., $\partial^{\alpha} = \partial^{\alpha_1}_{x_1} \dots  \partial^{\alpha_d}_{x_d} $
and $|\alpha|=\alpha_1+ \dots + \alpha_d$.

The set of bounded linear operators on a Hilbert space $\mc H$ is denoted by $\mc B(\mc H)$. For a closed linear operator $\mb L$ we write $\sigma(\mb L)$ and $\sigma_p(\mb L)$ for the spectrum and point spectrum, respectively.
Furthermore, we set $\mb R_{\mb L}(\lambda):=(\lambda-\mb L)^{-1}$ for $\lambda \notin \sigma(\mb L)$.

\section{The radial wave equation in similarity coordinates}

We restrict ourselves to radial solutions of Eq.~\eqref{Eq:NLW} and write $u(t,x) = u(t,r)$, where $r = |x|$, by slight abuse of notation. We introduce the radial Laplace operator on $\R^d$,
\begin{align*}
\Delta_r u(t,r) := \partial^2_r u(t,r) + \tfrac{d-1}{r} \partial_r u(t,r)
\end{align*}	 
and study the equation 
\begin{align}\label{Eq:Radial_NLW}
\partial^2_t u(t,r) -  \Delta_r u(t,r) = |u(t,r)|^{p-1} u(t,r).
\end{align}
The initial data at $t=0$ are assumed to be of the form
\begin{align}\label{Eq:Radial_NLW_Data}
u[0] = u_{T_0}[0] + (f,g),
\end{align}
where $T_0 > 0$ is fixed and $(f,g)$ can be chosen freely. At the origin we impose the natural boundary condition $\partial_r u(t,0) = 0$ for all $t > 0$. 
We define rescaled variables
	\begin{align*}
			U_1(t,r) := (T-t)^{\frac{2}{p-1}} u(t,r), \quad  U_2(t,r) := (T-t)^{1 + \frac{2}{p-1}} \partial_t u(t,r),
	\end{align*}
where $T > 0$. This yields the first order system
\begin{align*}
 \left ( \! \! \begin{array}{c} \partial_t U_1  \\ 						
			\partial_t U_2  \end{array} \! \! \right) =\left ( \! \!  \begin{array}{c} (T-t)^{-1}  U_2  - \frac{2}{p-1} (T-t)^{-1}   U_1 \\ 						
			(T-t) \Delta_r U_1  - (T-t)^{-1} \big (\frac{p+1}{p-1} U_2 - |U_1|^{p-1} U_1 \big )  \end{array} \! \!  \right)   
\end{align*}
with data 
\[ U_1(0,r) =(\tfrac{T}{T_0})^{\frac{2}{p-1}}  c_p + T^{\frac{2}{p-1}} f(r), \quad  U_2(0,r) = (\tfrac{T}{T_0})^{\frac{p+1}{p-1}} \tfrac{2}{p-1} c_p +  T^{\frac{p+1}{p-1}} g(r). \]
and boundary conditions
\[ \partial_r U_1(t,0) = \partial_r  U_2(t,0) = 0,\]
for all $t >0$.
In the rescaled variables the blow-up solution $u_T$ is static and corresponds to $\mb c_p := ( c_p ,\frac{2}{p-1} c_p )$. We introduce similarity variables
	 \[ \rho = \frac{r}{T-t}, \quad \tau = - \log (T-t) + \log T.\]
Derivatives transform according to
\[ \partial_t = \frac{e^{\tau}}{T} (\partial_{\tau} + \rho \partial_{\rho}), \quad \partial_{r} = \frac{e^{\tau}}{T} \partial_{\rho}.\] 
Setting  $\psi_j(\tau,\rho) := U_j(T-T e^{-\tau},T e^{-\tau} \rho)$ for $j=1,2$, we obtain the system
\begin{align}\label{Eq:First_order_css_nonl}
\left ( \! \!\begin{array}{c} \partial_{\tau} \psi_1  \\ 						
			\partial_{\tau} \psi_2  \end{array}  \! \! \right  ) =\left ( \! \!  \begin{array}{c}  \psi_2 - \rho \partial_{\rho} \psi_1 - \frac{2}{p-1} \psi_1 \\ 						
			\Delta_{\rho} \psi_1   - \rho \partial_{\rho} \psi_2 - \frac{p+1}{p-1}  \psi_2  + |\psi_1|^{p-1} \psi_1   \end{array} \! \!  \right)   
\end{align}
with the boundary conditions $\partial_{\rho} \psi_j(\tau,0) = 0$ and initial data 
\begin{align}\label{Eq:Data}
\begin{split}
\psi_1(0,\rho) & = (\tfrac{T}{T_0})^{\frac{2}{p-1}}  c_p + T^{\frac{2}{p-1}} f(T\rho ), \\
\psi_2(0,\rho)  & =  (\tfrac{T}{T_0})^{\frac{p+1}{p-1}} \tfrac{2}{p-1} c_p +  T^{\frac{p+1}{p-1}} g(T \rho).
\end{split}
\end{align}

We restrict the problem to the backward lightcone of $(T,0)$, i.e.,  we study Eq.~\eqref{Eq:First_order_css_nonl} for $\rho \in [0,1]$ and $\tau > 0$.

\subsection{Perturbations around the ODE blow up solution}
Inserting the ansatz
\[ \left ( \! \! \begin{array}{c} \psi_1  \\  \psi_2  \end{array}  \! \! \right) =  \left ( \! \!  \begin{array}{c}    \varphi_1 \\  \varphi_2 \end{array} \! \!   \right) + \mb c_p \]
into Eq.~\eqref{Eq:First_order_css_nonl} we obtain 
\begin{align*}
			 \left ( \! \!  \begin{array}{c} \partial_{\tau} \varphi_1  \\ 						
			\partial_{\tau} \varphi_2  \end{array} \! \!  \right) =\left ( \! \!  \begin{array}{c}  \varphi_2 - \rho \partial_{\rho} \varphi_1- \frac{2}{p-1} \varphi_1  \\ 						
			\Delta_{\rho} \varphi_1   - \rho \partial_{\rho} \varphi_2  - \frac{p+1}{p-1} \varphi_2  +  p c_p^{p-1} \varphi_1 +  N(\varphi_1)   \end{array} \! \!  \right)   
\end{align*}
for $\rho \in [0,1]$ and $\tau > 0$ where
\[ N(\varphi_1)=|c_p+\varphi_1|^{p-1}(c_p+\varphi_1)-c_p^{p}  -p c_p^{p-1} \varphi_1.\]
The initial data are given by
\begin{align*}
\begin{split}
\varphi_1(0,\rho) & = (\tfrac{T}{T_0})^{\frac{2}{p-1}}  c_p + T^{\frac{2}{p-1}} f(T\rho ) - c_p  \\
\varphi_2(0,\rho) & =  (\tfrac{T}{T_0})^{\frac{p+1}{p-1}} \tfrac{2}{p-1} c_p +  T^{\frac{p+1}{p-1}} g(T \rho)- \tfrac{2}{p-1} c_p.
\end{split}
\end{align*}

\section{Proof of Theorem \ref{Th:Main}}

Throughout this section, $p > \frac{d+3}{d-1}$ is a fixed real number.

\subsection{Functional setting}\label{Sec:FuncSet}

 We consider radial functions $\hat u$ defined on $\B_R$, i.e., $\hat u(\xi) = u(|\xi|)$, $\xi \in \B_R$. In order to avoid confusion owing to identification of $\hat u$ and $u$, we define 
\begin{align}\label{Def:SobRad}
\begin{split}
H_{r}^m(\B_R) := &  \{ u: (0,R) \to \C \text{ such that } u \text{ is } m- \text{times}   \\
&   \text{ weakly differentiable and }  \| u(| \cdot|) \|_{H^m(\B_R)} < \infty  \} ,
\end{split}
\end{align}
for $m \in \N_0$. The density of $C^{\infty}(\overline{\B_R})$ in $H^m(\B_R)$ implies the density of
\begin{align*}
C_e^{\infty}[0,R] := \{ u \in C^{\infty}[0,R] : u^{(2k+1)}(0)=0, k \in \N_0 \}
\end{align*}
in $H_{r}^m(\B_R)$. For the rest of the paper we set
\[ \boxed{m_d := \frac{d+1}{2} } \] 
and introduce the Hilbert space
\[\mc H := H_r^{m_d}\times H_r^{m_d-1}(\B)\]
with norm 
\begin{align*}
\| \mb u \|^2 := \| u_1(|\cdot|) \|^2_{H^{m_d}(\B)} +  \| u_2(|\cdot|) \|^2_{H^{m_d-1}(\B)},
\end{align*}
for $\mb u = (u_1,u_2)^T$. 

\subsubsection{Equivalent norms on $\mc H$}

We define a norm on $\mc H$ which is 'tailor-made' for the investigation of the linearized time evolution of the perturbation.
First, we need the following auxiliary result.

\begin{lemma}\label{Le:Aux_Norm}
Let $u \in H_r^{m_d}(\B)$. Then 
\begin{align*}
\| u(|\cdot|) \|^2_{H^{m_d}(\B)}  \simeq \| u \|^2_{L^2(0,1)} + \sum_{n=1}^{m_d} \| (\cdot)^{n-1} u^{(n)} \|^2_{L^2(0,1)}.
\end{align*}
Furthermore, for all $u \in H_r^{m_d-1}(\B)$,
\[\| u(|\cdot|) \|^2_{H^{m_d-1}(\B)}  \simeq  \sum_{n=0}^{m_d-1} \| (\cdot)^{n} u^{(n)} \|^2_{L^2(0,1)}.\]
\end{lemma}

The proof is given in Appendix \ref{Proof:Lemma_Aux_Norm}. To proceed, we define
\begin{align*}
		 \D u(\rho) := \left( \rho^{-1} \tfrac{d}{d\rho} \right)^{\frac{d-3}{2}} (\rho^{d-2} u(\rho)).
\end{align*}

Note that \begin{align}\label{Eq:DdSum}
\D u(\rho) = \sum_{n=0}^{m_d-2} a_n \rho^{n+1} u^{(n)}(\rho) = a_1 \rho u(\rho)  + \dots + \rho^{\frac{d-1}{2}} u^{(\frac{d-3}{2})}(\rho),
	\end{align}
for constants $a_n > 0$. The kernel of $\D$ consists of functions which are highly singular at the origin,
\begin{align*}
\mathrm{ker } \D =  \begin{cases}  \langle \rho^{-3} \rangle ,  &\quad \text{for } d=5, \\
				\langle \rho^{-3}, \rho^{-5}, \dots, \rho^{-(d-2)} \rangle, & \quad \text{for } d\geq 7. \\
 									\end{cases} 
\end{align*}
We also introduce the integral operator
 \begin{align*}
		\K u(\rho) := \rho^{2-d} \mc K^{\frac{d-3}{2}} u(\rho), \quad \mc Ku(\rho) := \int_0^{\rho} s u(s) ds.
\end{align*}

If $\mb u \in \mc H$, then $u_1,u_2$ are $m_d$--times, respectively, $(m_d-1)$--times weakly differentiable functions.
By Sobolev embedding,  $u_2 \in C^{m_d-2}[\delta,1]$ for every $0 < \delta <1$ and,  since $m_d > \frac{d}{2}$, $u_1 \in C[0,1] \cap C^{m_d-1}[\delta,1]$. Hence, the expressions $\D u_j$, $j=1,2$, are defined as sums of weighted classical derivatives on $(0,1]$. Furthermore, $\mathrm{ker} \D = \{0\}$ on $H^{m_d}_r(\B)$ and $H^{m_d-1}_r(\B)$, respectively. We infer that $\D$ is invertible on $\mc H$ and 
\begin{align*}
\K \D \mb u  = \D \K \mb u= \mb u.
\end{align*}

Now, consider the sesquilinear form $(\cdot| \cdot)_{D}$ defined by
\begin{align*}
 ( \mb u | \mb v )_{D}  & := \big (\D u_1 | \D v_1 \big )_{\dot H^2(0,1)} + \big (\D u_2| \D v_2 \big)_{\dot H^1(0,1)}  \\
 & + \big ([\D u_1]'(1) + [\D u_2](1) \big )   \overline{  \big  ([\D v_1]'(1)  + [\D v_2](1) \big)},
\end{align*}   
and set $\| \mb u \|_{D} := \sqrt{ ( \mb u | \mb u)_{D}}$. For $\mb u \in \mc H$, $(\D u_1)''$ and $(\D u_2)'$ have to be interpreted as sums of weighted weak derivatives. The proof of the next Proposition is provided in Appendix \ref{App:Equivalent_norms}. 

\begin{lemma}\label{Prop:EquivNorms}
$\| \cdot \|_{D}$ and $\| \cdot \|$ are equivalent norms on $\mc H$. In particular,
\begin{align*}
 \|\mb u \|^2 \lesssim \| \D u_1 \|^2_{H^2(0,1)} +  \| \D u_2 \|^2_{H^1(0,1)}   \lesssim \| {\mb u} \|^2_{D}  \lesssim \|\mb u \|^2
\end{align*}
for all $\mb u \in {\mc H} $.
\end{lemma}

By using the results of Lemma \ref{Le:Aux_Norm} and Lemma \ref{Prop:EquivNorms} together with the Sobolev embedding $H^m(0,1) \hookrightarrow C^{m-1}[0,1]$ and the density of $\Ceven^2$ in $\mc H$, we obtain the next result.

\begin{corollary}\label{Cor:PropDu}
Let $\mb u \in \mc H$. Then $\D u_1 \in C^1[0,1]$, $\D u_2 \in C[0,1]$ and 
\[  \D u_j(0) = 0 , \quad j =1,2. \]
\end{corollary}

\subsection{The linearized problem}

We exploit the following commutator relations satisfied by $\D$ and its inverse $\K$.

\begin{lemma}\label{Le:Comm_RelD}
Let $u \in H_r^{m_d-1}(\B) \cap C^{m_d-1}(0,1)$ and let $\Lambda$ denote the dilation operator defined by $\Lambda u(\rho) := - \rho u'(\rho)$.
Then
\begin{align}\label{Eq:Comm_Rel_LambdaD}
 \D \Lambda u(\rho) = \Lambda \D u(\rho)  + \D u(\rho) 
\end{align}
for $\rho \in (0,1)$. Furthermore, for $u \in H_r^{m_d}(\B) \cap C^{m_d}(0,1)$ we have the identity
\begin{align}\label{Eq:Comm_Rel_LaplaceD}
\D \Delta_{\rho} u(\rho) = [\D u]''(\rho)
 \end{align}
for $\rho \in (0,1)$. 
\end{lemma}

\begin{proof}
We note that the assumptions on $u$ imply that the above expressions vanish if and only if $u = 0$. To prove the identities, we proceed by induction. A direct calculation shows that for $d=5$, $D_5\Lambda u  = \Lambda D_5 u + D_5 u $. Assuming that Eq.~\eqref{Eq:Comm_Rel_LambdaD} is true for some odd number $d > 5$, we use the identities
\begin{align*}
(\cdot)^2 \Lambda u = \Lambda [(\cdot)^2 u] + 2 (\cdot)^2 u, \quad (\cdot)^{-1} [\Lambda u]' = \Lambda [(\cdot)^{-1} u'] - 2 (\cdot)^{-1} u',
\end{align*} to obtain 
\begin{align*}
D_{d+2}   \Lambda u(\rho)  & =   \rho^{-1} \big [D_{d}  [(\cdot)^2\Lambda u] \big ]'(\rho)   = \rho^{-1} \big [D_{d} \Lambda [(\cdot)^2 u] \big ]'(\rho)   + 2 D_{d+2} u(\rho)   \\
& = \rho^{-1} \big [\Lambda \D  [ (\cdot)^2 u] \big ]'(\rho)+ 3  D_{d+2} u(\rho)   = \Lambda D_{d+2} u(\rho) +  D_{d+2}  u(\rho).
\end{align*}
The identity given in Eq.~\eqref{Eq:Comm_Rel_LaplaceD} is well-known, cf.~\cite{Evans}, p.~75. 
\end{proof}

\begin{lemma}\label{Le:Comm_RelK}
Let $w \in C^1[0,1] \cap C^2(0,1)$ satisfy $w(0)=0$. Then
\begin{align}\label{Eq:Comm_Rel_LambdaK}
 \K \Lambda w = \Lambda \K w   - \K w ,
\end{align}
and
\begin{align}\label{Eq:Comm_Rel_LaplaceK}
\K w'' = \Delta_{\rho} \K w.
\end{align}
\end{lemma}

\begin{proof}
For $d=5$, Eq.~\eqref{Eq:Comm_Rel_LambdaK} follows from integration by parts. Assume that it is true for some $d >5$ odd. We use the identities
\begin{align*}
\mc K \Lambda w = \Lambda \mc K w + 2 \mc K, \quad (\cdot)^{-2} \Lambda w = \Lambda [(\cdot)^{-2} w] - 2 (\cdot)^{-2} w
\end{align*}
to infer that
\begin{align*}
K_{d+2} \Lambda w & = (\cdot)^{-2} K_d \mc K \Lambda w = (\cdot)^{-2} K_d  \Lambda \mc K  w + 2  K_{d+2}  w \\
& = (\cdot)^{-2} \Lambda  K_d \mc K w +  K_{d+2}  w = \Lambda K_{d+2} w -  K_{d+2} w.
\end{align*}
Using integration by parts one can easily check that Eq.~\eqref{Eq:Comm_Rel_LaplaceK} is true for $d=5$ provided that $w(0) = 0$. Assume that it holds for some odd number $d > 5$. To clarify notation we write $\Delta_{\rho}^{(d)}u(\rho) := \rho^{1-d} (\rho^{d-1} u'(\rho))'$. A straightforward calculation shows that
\begin{align*}
[K_{d+2} w''](\rho) & = \rho^{-d} \int_0^{\rho} s^{d-1} [K_d w''](s) ds = \rho^{-d} \int_0^{\rho} s^{d-1} [\Delta^{(d)}_{s} K_d w](s)  ds \\
& = \rho^{-1} [K_d w]'(\rho) = (2-d) \rho^{-d} \mc K^{\frac{d-3}{2}} w(\rho)  + \rho^{2-d} \mc K^{\frac{d-5}{2}} w(\rho) \\
& = \Delta^{(d+2)}_{\rho} K_d w(\rho).
\end{align*}

\end{proof}

\subsubsection{Well-posedness of the linearized time evolution}\label{Sec:LinWellPosed}

We define the operator $(\tilde{\mb L}_0, \mc D(\tilde{\mb L}_0))$ by 
\begin{align}\label{Def:L0}
  \tilde{\mb L}_0 \mb u(\rho):=  \left ( \begin{array}{c} u_2(\rho) + \Lambda u_1(\rho) - \frac{2}{p-1}  u_1(\rho) \\ 						
 \Delta_{\rho} u_1(\rho) + \Lambda u_2(\rho) - \frac{p+1}{p-1}  u_2(\rho)   \end{array} \right),
\end{align}
 \begin{align*}
& \mc D(\tilde{\mb L}_0) :=  \left \{  \mb u \in  \mc H \cap C^{\infty}(0,1)^2: \D u_2 \in C^2[0,1], \right. \\
& \left. \phantom{!!!!!!!!!!!!!!!!!!!!!!!!!!!!!!!} \D u_1 \in C^3[0,1],  [\D u_1]''(0) = 0   \right \}. 
\end{align*}

Using the results of Lemma \ref{Le:Comm_RelD} we get that
\begin{align}\label{Eq:CommRel_DL}
\D (\tilde{\mb L}_0 \mb u)_{j} = (\mb A_0 \D \mb u)_{j}
\end{align}
for $j=1,2$, where 
\[  \mb A_0 \mb w(\rho):=  \left ( \begin{array}{c} w_2(\rho) - \rho w'_1(\rho)  +  \frac{p-3}{p-1}  w_1(\rho) \\ 						
  w''_1(\rho) - \rho  w'_2(\rho)  - \frac{2}{p-1}  w_2(\rho)   \end{array} \right).\]
 In view of Lemma \ref{Prop:EquivNorms} is now obvious that the regularity properties satisfied by functions in $\mc D(\tilde{\mb L}_0)$ imply that $\tilde{\mb L}_0 \mb u \in \mc H$.
We note that $\Ceven^2\subset \mc D(\tilde{\mb L}_0)$, i.e.,  $\tilde{\mb L}_0$ is densely defined.

\begin{lemma}\label{Le:LumerPhillips_Est}
Let $\mb u \in \mc D(\tilde{\mb L}_0)$. Then
\begin{align*}
\mathrm{Re} ( \tilde{\mb L}_0  \mb u | \mb u)_{D} \leq  - \tfrac{2}{p-1} \| \mb u \|^2_{D}. 
\end{align*}
\end{lemma}

\begin{proof}
To abbreviate the notation we set $w_j := \D u_j$, $j=1,2$, for $\mb u \in \mc D(\tilde{\mb L}_0)$. By Eq.~\eqref{Eq:CommRel_DL} 
\begin{align*}
[\D( \tilde{\mb L}_0 \mb u)_1]'(\rho)+ [\D( \tilde{\mb L}_0 \mb u)_2](\rho) & = w'_2(\rho) - \rho  w''_1(\rho) + w''_1(\rho)  \\
& - \rho w'_2(\rho) - \tfrac{2}{p-1} \left ( w_1'(\rho)+ w_2(\rho) \right ).
\end{align*}
Evaluation at $\rho = 1$ yields 
\begin{align*}
 [\D( \tilde{\mb L}_0 \mb u)_1]'(1) + [\D( \tilde{\mb L}_0 \mb u)_2](1) = - \tfrac{2}{p-1} \left [ w_1'(1)+ w_2(1) \right ].
\end{align*}

Furthermore,
\begin{align*}
[\D( \tilde{\mb L}_0 \mb u)_1]''(\rho) =  w_2''(\rho)   + \Lambda w_1''(\rho)  - \tfrac{p+1}{p-1}  w_1''(\rho) ,  
\end{align*}
and 
\begin{align*}
[\D( \tilde{\mb L}_0 \mb u)_2]'(\rho)  =w_1'''(\rho)   + \Lambda w_2'(\rho)  - \tfrac{p+1}{p-1} w_2'(\rho) , 
\end{align*}
for $\rho \in (0,1)$.
Note that for functions $w \in C[0,1] \cap C^1(0,1)$ integration by parts yields
\[ \mathrm{Re} (\Lambda w| w)_{L^2(0,1)} = \tfrac{1}{2} \|w\|^2_{L^2(0,1)} - \tfrac{1}{2} |w(1)|^2.\]
With this we infer that
\begin{align*}
\begin{split}
\mathrm{Re}  \big ([\D( \tilde{\mb L}_0  \mb u)_1]'' \big |  [\D u_1]'' \big)_{L^2(0,1)} & = \mathrm{Re} \left (w_2''| w_1'' \right)_{L^2(0,1)}  - \tfrac{1}{2} \left |w_1''(1) \right |^2  \\
& - \left (\tfrac{1}{2} +\tfrac{2}{p-1} \right)  \left \|w_1'' \right \|^2_{L^2(0,1)},
\end{split}
\end{align*}
and 
\begin{align*}
\begin{split}
\mathrm{Re}  \big ([\D( \tilde{\mb L}_0  \mb u)_2]' \big |  [\D u_2]' \big )_{L^2(0,1)} & = \mathrm{Re} \left (w_1'''| w_2' \right)_{L^2(0,1)}  - \tfrac{1}{2} \left |w_2'(1) \right |^2  \\ & - \left (\tfrac{1}{2} +\tfrac{2}{p-1} \right)  \left \|w_2' \right \|^2_{L^2(0,1)}.
\end{split}
\end{align*}
Using these identities and performing one additional integration by parts we obtain
\begin{align*}
\mathrm{Re} ( \tilde{\mb L}_0 & \mb u| \mb u)_{D}= - \tfrac{2}{p-1} \big | [\D u_1]'(1)+ [\D u_2](1) \big|^2  \\
&  - \left (\tfrac{1}{2} +\tfrac{2}{p-1} \right) \big(  \big \|[\D u_1]''  \big \|^2_{L^2(0,1)} +    \big \|[\D u_2]'  \big \|^2_{L^2(0,1)} \big) \\
 &  - \tfrac{1}{2} \big | [\D u_1]''(1)-  [\D u_2]'(1) \big |^2 \leq  - \tfrac{2}{p-1} \| \mb u \|^2_{D}.
\end{align*}
\end{proof}

\begin{lemma}\label{Le:DenseRange}
Set $\mu = 1 - \tfrac{2}{p-1}$. For every $\mb f = (f_1,f_2)^T  \in \Ceven^2$ there exists a function $\mb u \in \mc D(\tilde{\mb L}_0)$ satisfying the equation
\[  (\mu - \tilde{\mb L}_0) \mb u =\mb  f. \]
\end{lemma}

\begin{proof}
Trivially, for $\mb f = 0$ we have $\mb u = 0$. Assume that $\mb f  \in \Ceven^2$ does not vanish identically. We set \[F(\rho) :=  \D f_2(\rho) + \D f_1(\rho) + \rho [\D f_1]'(\rho)\]
and define functions 
\begin{align*}
\begin{split}
w_1(\rho) & := \int_0^{\rho} \frac{1}{1-s^2} \int_s^1 F(s')ds' ds, \\
 w_2(\rho) & :=  \frac{\rho}{1-\rho^2}  \int_\rho^1 F(s)ds - \D f_1(\rho).
 \end{split}
\end{align*}
The properties of $F$ imply that $w_1 \in C^{\infty}(0,1) \cap C^3[0,1]$, $w_2 \in  C^{\infty}(0,1)  \cap C^2[0,1]$ and the functions satisfy the boundary conditions $w_1(0) = w_1''(0)  = w_2(0) = 0$.
A direct calculation shows that $w_1, w_2$ solve the system of equations
\begin{align}\label{Eq:LumPh1}
\begin{split}
 \rho w_1'(\rho)  - w_2(\rho)  & = \D f_1(\rho), \\
w_2(\rho) - w_1''(\rho) + \rho w_2'(\rho) & = \D f_2(\rho).
 \end{split}
\end{align}
We apply $\K$ to Eq.~\eqref{Eq:LumPh1} and use the results of Lemma \ref{Le:Comm_RelK} to obtain
\begin{align*}
\begin{split}
\K w_1(\rho) - \K w_2(\rho)  - \Lambda \K w_1(\rho) & = f_1(\rho), \\
2 \K w_2(\rho) - \Delta_{\rho} \K w_1(\rho)  - \Lambda \K w_1(\rho) & = f_2(\rho).
\end{split}
\end{align*}
Upon setting $u_j(\rho)  := \K w_j$ for $j=1,2$ and defining $\mb u := (u_1,u_2)^T$ we obtain a solution of the equation $(\mu - \tilde{\mb L}_0) \mb u = \mb f.$  The properties of the functions $w_j$ imply that $\mb u \in \mc D(\tilde{\mb L}_0)$ and the claim follows. 
\end{proof}

\begin{lemma}\label{Le:Compact_Pert}
The operator $\mb L': \mc H \to \mc H$ defined by
\begin{align}\label{Def:Perturbation}
 \mb L' \mb u :=  \left ( \begin{array}{c}  0 \\ pc_p^{p-1} u_1  \end{array} \right)
 \end{align}
is compact.
\end{lemma}

\begin{proof}
Let $(\mb u_n)_{n \in \N}$ be a sequence that is uniformly bounded in $\mc H$. By Lemma \ref{Prop:EquivNorms}, $(\D u_{1,n})_{n \in \N}$ is uniformly bounded in $H^2(0,1)$. The compact embedding $H^2(0,1) \hookrightarrow H^1(0,1)$ implies the existence of a subsequence, again denoted by $(\D u_{1,n})_{n \in \N}$, which is a Cauchy sequence in $H^1(0,1)$. The claim follows from the fact that 
\[  \|\mb L' \mb u_n - \mb L' \mb u_m\| \lesssim  p c_p^{p-1} \| \D u_{1,n} -\D u_{1,m} \|_{H^1(0,1)}. \]
\end{proof}

In view of Lemma \ref{Le:LumerPhillips_Est}, Lemma \ref{Le:DenseRange} and the boundedness of $\mb L'$, we can apply the Lumer-Phillips Theorem \cite{engel}, p.~83, Theorem 3.15, together with the Bounded Perturbation Theorem \cite{engel}, p.~158, to show that the linearized time evolution is well-posed. In particular, by the equivalence of $\| \cdot \|_{D}$ and $\| \cdot \|$ we can formulate the following result.

\begin{proposition}\label{Prop:Semigroup}
The operator $(\tilde{\mb L}_0, \mc D(\tilde{\mb L}_0))$ is closable and its closure, denoted by $(\mb L_0, \mc D(\mb L_0))$, generates a strongly-continuous
 one-parameter semigroup of bounded operators $(\mb S_0(\tau))_{\tau \geq 0}$ on $\mc H$ satisfying the growth estimate
\begin{align*}
\|\mb S_0(\tau )\mb u \| \leq  M e^{-\frac{2}{p-1} \tau} \|\mb u\|
\end{align*}
for all $\mb u \in \mc H$, $\tau > 0$ and a constant $M \geq 1$.  Furthermore, the operator 
\[ \mb L:= \mb L_0 + \mb L', \quad \mc D(\mb L) =\mc D(\mb L_0),\]
is the generator of a strongly-continuous semigroup $(\mb S(\tau))_{\tau \geq 0}$.
\end{proposition}

In order to derive a suitable growth estimate for $\mb S(\tau)$ we investigate the spectrum of the operator $(\mb L, \mc D(\mb L))$.
\subsubsection{Spectral properties of the generator}

\begin{lemma}\label{Le:Spectrum}
Let $\lambda \in \sigma(\mb L)$. Then either $\lambda = 1$ or 
\begin{align*}
  \mathrm{Re }\lambda  \leq \mathrm{max } \{- \tfrac{2}{p-1}, -1\} .
\end{align*}
Moreover, $\lambda = 1$ is an eigenvalue and the corresponding eigenspace is spanned by the constant function $\mb g = (1, \frac{p+1}{p-1})$.
\end{lemma}

\begin{proof}
Let $\lambda \in \sigma(\mb L)$. If $\mathrm{Re }\lambda  \leq - \frac{2}{p-1}$ then the assertion is obviously true. So assume that $\mathrm{Re }\lambda  > - \frac{2}{p-1}$. Then $\lambda \notin \sigma(\mb L_0)$ by standard semigroup theory. The identity $(\lambda - \mb L)= (1- \mb L' \mb R_{\mb L_0}) (\lambda - \mb L_0) $ and the compactness of $\mb L'$ imply that $\lambda \in \sigma_p(\mb L)$. In particular, there exists an eigenfunction $\mb u \in \mc D(\mb L)$ satisfying the eigenvalue equation $(\lambda - \mb L) \mb u = 0$. The regularity properties of functions in $\mc H$ imply that $\mb L$ acts as a classical differential operator on the interval $(0,1)$. By a straightforward calculation one can check that if $\mb u$ satisfies the eigenvalue equation then $u_1 \in C[0,1] \cap  C^{m_d-1}(0,1)$ is a nontrivial solution of the second order ordinary differential equation
\begin{align}\label{Eq:Eigenvalue}
\begin{split}
\rho^2 u''(\rho)   & -\Delta_{\rho} u(\rho)  +  2(\lambda +\tfrac{p+1}{p-1} ) \rho  u'(\rho)  \\
& + (\lambda +\tfrac{2}{p-1})(\lambda +\tfrac{p+1}{p-1} ) u(\rho) - p c_p^{p-1} u(\rho) = 0.
\end{split}
\end{align}
Since the coefficients are smooth on $(0,1)$ we infer that $u_1 \in C^{\infty}(0,1) \cap C[0,1]$. 
We apply $\D$ to the equation and use the results of Lemma \ref{Le:Comm_RelD}, where we proved the identity
\[ \D [\rho u'(\rho)] = \rho  [\D u]'(\rho) -\D u(\rho). \]
Similarly, one can show that 
\begin{align*}
\D [\rho^2 u''(\rho)] = \rho^2 [\D u]''(\rho) - 2 \rho [\D u]'(\rho) + 2 \D u(\rho).
\end{align*}
Upon setting $w := \D u_1$ we infer that 
\begin{align}\label{Eq:EigenvalueD}
\begin{split}
-(1- & \rho^2) w''(\rho)  + 2\rho (\lambda +\tfrac{2}{p-1} )w'(\rho)  \\
& + (\lambda +\tfrac{2}{p-1} - 1)(\lambda +\tfrac{2}{p-1}) w(\rho) - p c_p^{p-1} w(\rho) = 0,
\end{split}
\end{align}
where $w \in H^2(0,1) \cap C^1[0,1]$ satisfies the boundary condition $w(0) = 0$, cf.~Corollary \ref{Cor:PropDu}. By substituting $\rho \mapsto z:= \rho^2$ and setting $v(z):=w(\sqrt{z})$ one obtains the hypergeometric differential equation 
\begin{equation*}
z(1-z)v''(z)+[c-(a+b+1)z]v'(z)-abv(z)=0 
\end{equation*} 
with parameters
\begin{align*}
a= \tfrac12 (\lambda -2), \quad b=\tfrac12 (\lambda + \tfrac{p+3}{p-1}), \quad c=\tfrac12.
\end{align*}
The assumption $\mathrm{Re} \lambda > - \frac{2}{p-1}$ implies that $\mathrm{Re}(c-a-b) = 1 - \frac{2}{p-1} - \mathrm{Re} \lambda < 1$. Let us assume for the moment that $\mathrm{Re}(c-a-b)$ is not zero or a negative integer. Around $\rho = 1$ two linearly independent solutions are given by $\{v_1,\tilde v_1\}$, where 
\begin{align*}
\begin{split}
v_1(z)  & = {}_2F_1(a, b; a+b+1-c;1-z) \\
\tilde v_1(z) & =(1-z)^{c-a-b} {}_2F_1(c-a,c-b; 1+c-a-b;1-z),
\end{split}
\end{align*}
and ${}_2F_1$ denotes the standard hypergeometric function, see e.g.~\cite{DLMF}. If $\mathrm{Re}(c-a-b) = -n$, for $n \in \N_0$, then one solution is still given by $v_1$ and \[ \tilde v_1(z) = c v_1(z) \log(1-z) + (1-z)^{-n} h(z),\]
where $c$ might be zero for $n \in \N$ and $h$ is analytic around $z =1$. In all cases, the requirement $w \in H^2(\frac{1}{2},1)$ excludes the  solution $\tilde v_1$ and we infer that $v$ is a multiple of  $v_1$. Around $\rho =0$ we have the fundamental system  $\{v_0,\tilde v_0\}$, where
\begin{align*}
\tilde{v}_0(z)& = {}_2F_1(a,b;c;z), \\
v_0(z)& = z^{1/2}{}_2F_1(a+1-c,b+1-c; 2-c; z).
\end{align*} 
Hence, there are constants $c_0$, $c_1 \in \C$ such that 
\[v_1=c_0 \tilde{v}_0+c_1 v_0. \]
The condition $w(0)=0$ implies that $v(0)=v_1(0)= 0$ and thus $c_0$ must be zero.  By \cite{DLMF},
\[c_0=\frac{\Gamma(a+b+1-c)\Gamma(1-c)}{\Gamma(a+1-c)\Gamma(b+1-c)}.\]
Since the gamma function has no zeros, $c_0$ can only vanish if either $a+1-c$ or $b+1-c=0$ is a pole. This is equivalent to
\begin{align*}
\lambda = 1 - 2k \quad \text{or} \quad  \lambda = -2k -\tfrac{2p}{p-1}-\tfrac{2}{p-1}  \quad \text{for} \quad k \in \N_0.
\end{align*}
The latter condition implies that $\lambda < -\frac{2}{p-1}$ which is excluded by assumption. 
The first condition yields that $\lambda = 1 - 2k$ for some $k \in \N_0$, hence 
\[ \lambda  \in \{1,-1,-3, \cdots\}.\]
For $p > 3$, $-\frac{2}{p-1} \in (-1,0)$, hence $\lambda = 1$ is the only possibility. For $\frac{d+3}{d-1} < p \leq 3$, $- \infty < -\frac{2}{p-1} \leq -1$ and in this case either $\lambda = 1$ or $\lambda \leq -1$. This proves the first claim.

A straightforward calculation shows that $\mb g = (1, \frac{p+1}{p-1})$ satisfies the equation 
\[(1 - \mb L) \mb g = 0.\]
Furthermore, it is easy to check that $\D g_1  = \alpha_1 \rho$, $\D g_2  = \alpha_2 \rho$  for constants $\alpha_1,\alpha_2 > 0$. Hence, $\mb g \in \mc D( \mb {\tilde L_0})$, which proves that 1 is an eigenvalue. Suppose that there is another eigenfunction $\tilde {\mb g} \in \mc H$ associated to $\lambda = 1$. Then $\D  \tilde g_1$ satisfies Eq.~\eqref{Eq:EigenvalueD}. With the same arguments as before we infer that  $\D  \tilde g_1(\rho) = \tilde  \alpha \rho \cdot {}_2F_1(a+1-c,b+1-c; 2-c; \rho^2)$, for some $\tilde  \alpha \in \C$. For $\lambda = 1$, $a+1-c = 0$, hence $\D  \tilde g_1 =  \tilde  \alpha \rho$ which implies that $\tilde g_1 = \beta g_1$ for some $\beta \in \C$. The equation $(1 - \mb L) \tilde {\mb g} = 0$ then shows that $\tilde g_2 = \beta g_2$, which proves that the eigenspace of $\lambda = 1$ is spanned by $\mb g$.
\end{proof}

\subsubsection{Time evolution for the linearized problem}

\begin{lemma}\label{Le:LinearTimeEvol}
There exists a projection $\mb P \in \mc B(\mc H)$ onto $\langle \mb g \rangle$ 
which commutes with $\mb S(\tau)$  and 
\[ \mb S(\tau) \mb P \mb f = e^{\tau} \mb P \mb f  \]
for all $\mb f \in \mc H$ and all $\tau > 0$. Moreover,
\begin{align}\label{Eq:StableEst}
 \|(1 - \mb P) \mb S(\tau) \mb f \|  \leq M e^{-\mu_p \tau}  \|(1 - \mb P) \mb f  \| 
\end{align}
for all $\mb f \in \mc H$, $\tau > 0$, some constant $M \geq 1$ and
\[ \mu_p = \min\{\tfrac{2}{p-1}, 1 \} - \varepsilon \]
for some small $\varepsilon > 0$.
\end{lemma}

\begin{proof}
The eigenvalue $\lambda = 1$ is isolated and we define $\mb P \in \mc B(\mc H)$ by 
\begin{align*}
\mb P=\frac{1}{2\pi i}\int_\gamma \mb R_{\mb {L}}(\lambda) d\lambda,
\end{align*}
where $\gamma$ is a positively oriented circle around $1$ in the complex plane with radius $r_{\gamma} = \frac12$, cf.~\cite{kato}, p.~178, 
Theorem 6.5. The projection commutes with the operator $\mb L$ and its resolvent, see \cite{kato}, p.~173, Theorem 6.5, and thus with the semigroup. Furthermore,  $\mc H = \ker \mb P \oplus\rg \mb P$ and the operator $\mb L$ is decomposed into parts $\mb L|_{\mc D(\mb L) \cap \ker \mb P}$ and $\mb L|_{\mc D(\mb L) \cap \rg \mb P}$, where 
$\mb L|_{\mc D(\mb L) \cap \rg \mb P} \mb u =  \mb L \mb u$ for 
$\mb u \in \mc D(\mb L) \cap \rg \mb P$ (analogously for $\mb L|_{\mc D(\mb L) \cap \ker \mb P}$). The spectrum of the restricted operator is given by
\[\sigma(\mb L|_{\mc D(\mb L) \cap \ker \mb P}) = \sigma(\mb L) \setminus \{1\}, 
\quad \sigma(\mb L|_{\mc D(\mb L) \cap \rg \mb P}) = \{1\}.\]

It is immediate that $ \langle \mb g \rangle \subset \rg \mb P$.  It remains to show the reverse inclusion. We first observe that if $\mathrm{dim}~\rg \mb P= \infty$, then $\lambda = 1$ would belong to the essential spectrum of $\mb L$  \cite{kato}, p.~239, Theorem 5.28, which is invariant under compact perturbations \cite{kato}, p.~244, Theorem 5.35. However, $1 \not \in \sigma(\mb L_0)$ and we infer that $\mb P$ has finite rank. 

Next, we convince ourselves that  $\rg \mb P  \subset  \mc D(\mb L)$. Let $\mb u \in \mathrm{rg} \mb P$. By density of $\mc D(\mb L)$ in $\mc H$,  there exists a sequence $(\mb u_n)_{n \in \N_0} \subset \mc D(\mb L)$ such that $\mb u_n \to \mb u$.  The fact that $\mb P$ is bounded yields $\mb P \mb u_n \to \mb u$ and since $\mb P \mc D(\mb L) \subset \mc D(\mb L)$ by \cite{kato}, p.~178, Theorem 6.17, $(\mb P \mb u_n)_{n \in \N_0}  \subset \rg \mb P \cap  \mc D(\mb L)$. By boundedness of $\mb L|_{\mc D(\mb L) \cap \rg \mb P }$ we get that $\mb L \mb P \mb u_n  \to \mb f$, for some $\mb f \in \rg \mb P \cap  \mc D(\mb L)$. The closedness of $\mb L$ now implies that $\mb L \mb P \mb u_n \to \mb L \mb u$ and $\mb u \in  \mc D(\mb L)$.  
We infer that  $1 - \mb L|_{\rg \mb P}$ acts on a finite dimensional Hilbert space and that $\lambda = 0$ is its only spectral point. Hence, it is  nilpotent and  $(1 - \mb L|_{\rg \mb P})^k \mb u = 0$ for all $\mb u \in  \rg \mb P$ and some minimal $k \in \N$. If $k=1$, then the claim follows. So let us assume that $k \geq  2$. Then there exists a nontrivial function $\mb u \in \rg \mb P \subset \mc D(\mb L)$ such that $(1 - \mb L|_{\rg \mb P}) \mb u \in  \ker(1 - \mb L|_{\rg \mb P}) \subset \ker(1 - \mb L)$, i.e., $\mb u$ satisfies the equation 
\begin{align*}
(1-\mb L) \mb u = \alpha \mb g
\end{align*}
for some $\alpha \in \C$. A straightforward calculation shows that the first component then satisfies 
\[ \rho^2 u_1''(\rho) - \Delta_{\rho} u_1(\rho) + \tfrac{4 p}{p-1} \rho u_1'(\rho) = \tfrac{3p+1}{p-1} \alpha.\]
Since  $u_1 \in C[0,1] \cap C^{m_d-1}(0,1)$ for $\mb u \in \mc H$ the equation can be interpreted in a classical sense for $\rho \in (0,1)$. Smoothness of the coefficients implies that $u_1 \in C^{\infty}(0,1)$. We apply $\D$ and set $w :=\D u_1$, where $w \in H^2(0,1) \cap C^{\infty}(0,1)$ and $w(0)=0$ by Lemma \ref{Cor:PropDu}. This yields
\begin{align*}
-(1-\rho^2) w''(\rho) + \tfrac{2(p+1)}{p-1} \rho w'(\rho) - \tfrac{2(p+1)}{p-1} w(\rho) = g(\rho) ,
\end{align*}
with $ g(\rho) = \tilde \alpha_p \rho$ for some $\tilde \alpha_p \in \C$, $\tilde \alpha_p \neq 0$. Recalling the proof of Lemma \ref{Le:Spectrum} we know  that a fundamental system is given by $w_0(\rho) = \rho$ and 
\[ w_1(\rho)= {}_2F_1(-\tfrac12 ,\tfrac{p+1}{p-1},\tfrac12;\rho^2)=(1-\rho^2)^{-\frac{2}{p-1}} h(\rho),\]
where $h$ is continuous on $[0,1]$ and $h(0)\neq 0$ as well as $\lim_{\rho \to 1} h(\rho) \neq 0$. For the Wronskian we obtain $W(w_0,w_1) = -(1-\rho^2)^{-\frac{p+1}{p-1}}$. From the variation of constants formula and the boundary condition $w(0)=0$ we infer that 
\begin{align*}
 w(\rho)=c_0 \rho -   \tilde \alpha_p \rho  \int_{\rho_0}^\rho s h(s) ds +  \tilde \alpha_p (1-\rho^2)^{-\frac{2}{p-1}} h(\rho)\int_{0}^\rho s^2 (1-s^2)^{\frac{2}{p-1}}  ds 
\end{align*}
for some constants $c_0 \in \C$ and $\rho_0 \in  [0,1]$. By continuity of  $w$ it is required that $\lim_{\rho \to 1} (\int_{0}^\rho s^2 (1-s^2)^{\frac{2}{p-1}}  ds ) =0$, which is impossible since the integrand is strictly positive. This proves that $k=1$.

\medskip

Finally, we establish the estimates for the semigroup. Recall that the growth bound $\omega_0(\mc S )$, cf.~\cite{engel}, p.~251, for a semigroup $\mc S=(\mb S(\tau) )_{\tau >0}$ can be related to the spectral radius $ r(\mb S(\tau))$ of the bounded operator $\mb S(\tau)$ for each $\tau > 0$ by the Hadamard formula. This yields $\omega_0(\mc S ) = \frac{1}{\tau} \log r(\mb S(\tau))$.  From Lemma \ref{Prop:Semigroup} we know that $r(\mb S_0(\tau)) \leq e^{-\frac{2}{p-1} \tau}$ for all $\tau > 0$. 
By the Duhamel formula, see \cite{engel}, p.~258, Prop.~2.12, 
\begin{align*}
(1 - \mb P)\mb S(\tau) = \mb S_0(\tau) + \int_0^{\tau} \mb S_0(\tau - \tau' ) \mb L' \mb S(\tau') d\tau' - \mb P \mb S(\tau). 
\end{align*}
Compactness of $\mb L'$ and the fact that $\mb P$ has finite rank imply that for every $\tau > 0$ the operator  $(1 - \mb P)\mb S(\tau)$ is the sum of $\mb S_0(\tau)$ and a compact perturbation. If $r((1 - \mb P)\mb S(\tau)) \leq e^{- \frac{2}{p-1} \tau}$ for all $\tau >0$, Eq.~\eqref{Eq:StableEst} follows immediately. If $r((1 - \mb P)\mb S(\tau)) > e^{- \frac{2}{p-1} \tau}$, then $(1 - \mb P)\mb S(\tau)$ has a spectral point $\mu \in \C$ with $|\mu| =r((1 - \mb P)\mb S(\tau)) = e^{(- \frac{2}{p-1} + \alpha) \tau}$ for some $\alpha > 0$. Since $\mu$ is not in the spectrum of $\mb S_0(\tau)$ it must be an eigenvalue and by the spectral mapping theorem for the point spectrum \cite{engel}, IV.3.7, p.~277, the generator has an eigenvalue $\lambda$ with $\mathrm{Re} \lambda =- \frac{2}{p-1} + \alpha$.  In view of the spectrum of $\mb L$ on the stable subspace, this is a contradiction if $p \geq 3$. If $p < 3$, then we know that $\mathrm{Re} \lambda \leq -1$ and we infer that 
 $|\mu| \leq  e^{-\tau}$. This implies that $r((1 - \mb P)\mb S(\tau)) \leq e^{-\omega_p\tau}$ for all $\tau > 0$,  where $\omega_p = \min\{\frac{2}{p-1}, 1 \}$. This and the  definition of the growth bound show that for every $\varepsilon > 0$ and $\mu_p := \omega_p  - \varepsilon$ there is a  constant $M \geq 1$ such that
\[ \|(1 - \mb P) \mb S(\tau) \mb f \|  \leq M e^{-\mu_p \tau}  \|(1 - \mb P) \mb f  \| \]
for all $\mb f \in \mc H$. The fact that $\lambda =1$ is an eigenvalue with eigenfunction $\mb g$ yields $\mb S(\tau) \mb P \mb f = e^{\tau} \mb P \mb f $.
\end{proof}

\subsection{Nonlinear Perturbation Theory}

For the rest of this section we restrict ourselves to real valued functions. Furthermore, whenever the domain in the $H_r^{m_d}-$norm is not indicated, it is the unit ball $\B \subset \R^d$.  By $\mc B_{\delta}$ we denote the open ball in $\mc H$ centered at the origin with radius $\delta > 0$.

\subsubsection{Estimates for the nonlinearity}
For $\mb u = (u_1,u_2)$ we define
\[ \mb N(\mb u) (\rho) := \left ( \begin{array}{c} 0 
\\  N(u_1(\rho)) \end{array} \right ),\] 
where 
\begin{align*}
N(x) := n(c_p + x) - n(c_p) - n'(c_p) x, \quad n(x) = x |x|^{p-1}, 
\end{align*}
and $c_p$ is the constant from Eq.~\eqref{Eq:BlowUpSol}. Obviously, $N(0) = N'(0) = 0$.

\begin{lemma}\label{Le:EstNonlinearity}
Let  $\delta > 0$ be sufficiently small. Then 
\begin{align}\label{Eq:Nonlin_Scalar_Lip}
\| \mb N(\mb u) - \mb N(\mb v) \|  \lesssim( \|\mb u \|  +\|\mb v \| )  \|\mb u  -\mb v \|
\end{align}
for all $\mb u, \mb v \in  \mc B_{\delta} \subset \mc H$.
\end{lemma}

\begin{proof}
 We show that 
\begin{align}\label{Eq:NonlinHelp}
\begin{split}
& \| N(u(|\cdot|))   - N(v(|\cdot|) )  \|_{H^{m_d}}  \\
& \lesssim \big ( \|u(|\cdot|)  \|_{H^{m_d}}   +\|v(|\cdot|  )  \|_{H^{m_d}} \big )   \| u(|\cdot|)  - v(|\cdot|)  \|_{H^{m_d}} 
\end{split}
\end{align}
for all $u,v \in \Ceven$ that have $H_r^{m_d}$--norm less then $\delta$. By density, this estimate can be extended to all of $H_r^{m_d}(\B)$ and Eq.~\eqref{Eq:Nonlin_Scalar_Lip} follows.  

Note that $c_p > \frac{3}{4}$ for all $p > 1$. In fact, $c_p \to 1$ for $p \to \infty$ and $c_p \to \infty$ as $p \to 1$. This implies that $N$: $[-\frac12,\frac12] \to \R$ is smooth. By the fundamental theorem of calculus,
\begin{align}\label{Eq:Diff_Nonl_Scalar}
\begin{split}
N(x) - N(y) & = \int_{y}^{x} N'(s) ds  \\
&  = (x-y) \int_0^1 N'(y+s(x-y))  ds, \\
\end{split}
\end{align}
for all $x, y \in   [-\frac12,\frac12]$. From Sobolev embedding we know that
\[ \|u\|_{L^{\infty}(0,1)} \lesssim \| u(|\cdot|) \|_{H^{m_d}} \]
for all $u \in H^{m_d}(\B)$. Hence, we choose $\delta > 0$ so small that $\|u\|_{L^{\infty}(0,1)} < \frac12$ for all $ \| u(|\cdot|) \|_{H^{m_d}} < \delta$. Now let $u,v \in  \Ceven$ satisfy this smallness condition. Then $v(\rho) + s(u(\rho)-v(\rho)) \in [-\frac12,\frac12]$ for all $s, \rho \in [0,1]$. The  fact that $H^{m_d}(\B)$ is a Banach algebra and Eq.~\eqref{Eq:Diff_Nonl_Scalar} imply that
\begin{align*}
 & \| N  (u(|\cdot|))    - N(v(|\cdot|) )  \|_{H^{m_d}}   \\
 & \leq   \| u(|\cdot|) - v(|\cdot|)   \|_{H^{m_d} } \int_0^1  \big \|  \big [N'  \circ \big (v    + s (u - v)) \big ](|\cdot|)  \big  \|_{H^{m_d}} ds.
\end{align*}
We estimate the integral term with Moser's inequality, see for example \cite{Rauch}. To this end, we extend the relevant functions to the whole space. Using a smooth cut-off function we can construct $F$: $\R \to \R$ such that $F$ is smooth, $F = N'$ on $[-\frac12,\frac12]$ and $F = 0$ on $\R \setminus [-\frac{3}{4},\frac{3}{4}]$. The properties of $N$ imply that $F(0) = 0$. To extend $u$ and $v$ we apply Lemma \ref{Le:Extension_Bounds} and note that the extension $ U \in C^{m_d}[0,\infty)$ of $u$ can always be constructed in such a way that 
\begin{align*}
\|  U(|\cdot|) \|_{L^{\infty}(\R^d)}  = \| u \|_{L^{\infty}(0,1)}.
\end{align*}
By Lemma \ref{Le:Extension_Bounds},
\[ \| U(|\cdot|)\|_{H^{m_d}(\R^d)}  \lesssim \| u(|\cdot|) \|_{H^{m_d}(\B)}.\] 
The respective extension for $v$ is denoted by $V$. By Moser's inequality,
\begin{align*}
 \big \|  \big [N'    & \circ \big (v    + s (u - v))\big ](|\cdot|)   \big  \|_{H^{m_d}(\B)}    \\
&  \leq \big \|  \big [F  \circ \big (U    + s (U - V))\big ](|\cdot|)  \big  \|_{H^{m_d}(\R^d)}   \\
&   \lesssim  \big \|  U(|\cdot|)    + s  \big (U(|\cdot|) - V(|\cdot|) \big)  \big \|_{H^{m_d}(\R^d)}   \\
&  \lesssim \|U(|\cdot|)  \|_{H^{m_d}(\R^d)} + \|V(|\cdot|)  \|_{H^{m_d}(\R^d)}  \\
& \lesssim \|u(|\cdot|)  \|_{H^{m_d}(\B)} + \|v(|\cdot|)  \|_{H^{m_d}(\B)} 
\end{align*}
for all $s \in [0,1]$. This implies Eq.~\eqref{Eq:NonlinHelp}.
\end{proof}

\subsubsection{The nonlinear Cauchy problem}\label{Sec:NonlinearCP}
For $\mb u \in \mc H$ we consider integral equation
\begin{align*}
\Phi(\tau) = \mb S(\tau) \mb u + \int_0^{\tau} \mb S(\tau - \tau')  \mb N(\Phi(\tau')) d\tau'
\end{align*}
on the Banach space 
\[ \mc X := \{ \Phi \in C([0,\infty), \mc H): \| \Phi  \|_{\mc X} := \sup_{\tau > 0} e^{\mu_p \tau} \| \Phi(\tau) \|  < \infty \}.\]
Here, $\mu_p >0$ is the constant from Lemma \ref{Le:LinearTimeEvol}.  In the following  we denote by $\mc X_{\delta}$ the closed subspace
\[ \mc X_{\delta} := \{\Phi \in \mc X: \| \Phi  \|_{\mc X} \leq \delta \}.\]

\subsubsection{Correction of the unstable behavior}

We define
\begin{align*}
\mb C(\Phi, \mb u) := \mb P \mb u + \int_0^{\infty} e^{-\tau'} \mb P  \mb N(\Phi(\tau')) d\tau',
\end{align*}

and set
\begin{align*}
\mb K(\Phi, \mb u)(\tau) :=\mb S(\tau)\mb u  + \int_0^{\tau} \mb S(\tau - \tau')  \mb N(\Phi(\tau')) d\tau' - e^{\tau} \mb C(\Phi, \mb u).
\end{align*}

\begin{theorem} \label{Th:GlobalEx_ModEq}
Let $\delta > 0$ be sufficiently small and let $c > 0$ be sufficiently large (independent of $\delta$). For every $\mb u\in \mc H$ with $\| \mb u \| \leq \frac{\delta}{c}$ there exists a unique $\mb \Phi( \mb u) \in \mc X_{\delta}$ that solves the equation
\[ \mb \Phi( \mb u)   =\mb K (  \mb \Phi( \mb u) , \mb u). \]
Furthermore, the map $\mb u \mapsto \mb \Phi(\mb u)$ is continuous.
\end{theorem}

\begin{proof}
We argue along the lines of \cite{DonSch12}, \cite{DonSch13}. For fixed $(\Phi, \mb u) \in \mc X_{\delta} \times \mc H$, continuity of the map $\tau \mapsto \mb K(\Phi, \mb u)(\tau)$: $[0,\infty) \to \mc H$ follows essentially from the strong continuity of the semigroup. To see that $\mb K(\cdot , \mb u)$ maps $\mc X_{\delta}$ into itself for $\|\mb u\| \leq \frac{\delta}{c}$, we decompose the operator according to 
\[ \mb K(\Phi, \mb u) =  \mb P \mb K(\Phi, \mb u)+(1- \mb P) \mb K(\Phi, \mb u).\]
By Lemma \ref{Le:EstNonlinearity} we have
\[ \| \mb N(\Phi(\tau))\|   \lesssim \delta^2 e^{-2\mu_p \tau}, \]
for $\Phi \in \mc X_{\delta}$ and all $\tau >0$. Hence,
\begin{align*}
\|\mb P\mb K(\Phi, \mb u)(\tau)\| \lesssim  e^{\tau} \int_{\tau}^{\infty} e^{-\tau'} \| \mb P \mb N(\Phi(\tau')) \| d\tau'  \lesssim \delta^2  e^{-2\mu_p \tau},
\end{align*}
and 
\begin{align*}
\begin{split}
\|(1-\mb P) \mb K(\Phi, \mb u)(\tau) \| &  \lesssim   e^{- \mu_p \tau} \| \mb u \| +  \int_{0}^{\tau} e^{- \mu_p(\tau -\tau')} \|\mb N(\Phi(\tau')) \| d\tau'  \\ &  \lesssim ( \tfrac{\delta}{c} +  \delta^2) e^{- \mu_p \tau}.
\end{split}
\end{align*}
Consequently, $  e^{\mu_p \tau} \|\mb K(\Phi, \mb u)(\tau)\| \lesssim \tfrac{\delta}{c} +  \delta^2 \leq \delta $
for all $\tau > 0$, given that $c > 0$ is sufficiently large and $\delta > 0$ is sufficiently small. For the contraction property of $\mb K(\cdot, \mb u)$
we use a similar decomposition and the fact that
\[ \| \mb N(\Phi(\tau)) - \mb N(\Psi(\tau))  \|   \lesssim   \delta e^{-\mu_p \tau}  \| \Phi(\tau) - \Psi(\tau)\| \]
for any $\Phi,\Psi \in \mc X_{\delta}$ and all $\tau >0$ by Lemma \ref{Le:EstNonlinearity}. In particular,
\begin{align*}
\|\mb P[ \mb K(\Phi, \mb u)(\tau) - \mb K(\Psi, \mb u)(\tau)] \| &  \lesssim  \delta \| \Phi - \Psi \|_{\mc X}   \int_{\tau}^{\infty} e^{\tau -\tau'(1+2 \mu_p)}  d\tau'   \\
& \lesssim \delta  e^{-2\mu_p \tau} \| \Phi - \Psi \|_{\mc X},
\end{align*}
and
\begin{align*}
\|(1 -\mb P) [\mb K(\Phi, \mb u)(\tau) - \mb K(\Psi, \mb u)(\tau)]\| & \lesssim \delta e^{- \mu_p \tau} \int_{0}^{\tau} \| \Phi(\tau') - \Psi(\tau')\| d\tau'  \\& \lesssim \delta  e^{-\mu_p \tau} \| \Phi - \Psi \|_{\mc X},
\end{align*}
which implies that $\mb K(\cdot, \mb u)$ is contracting given that $\delta$ is sufficiently small. An application of the  Banach fixed point theorem yields the existence of a unique solution $\mb \Phi(\mb u) \in \mc X_{\delta}$. Continuity of the solution map $\mb u \mapsto \mb \Phi(\mb u)$ follows easily from the estimate
\[   \| \mb K( \Phi, \mb u)(\tau)  -\mb K( \Phi, \mb v)(\tau)  \| = \| \mb S(\tau)(1- \mb P)(\mb u- \mb v) \| = e^{- \mu_p \tau} \| \mb u - \mb v\| \]
and the fact that $\mb K(\cdot, \mb u)$ is a contraction mapping.
\end{proof}

\subsubsection{The initial data operator}

For $R > 0$ we set \[\mc H^R := H_r^{m_d}\times H_r^{m_d-1}(\B_R)\] with norm defined by 
\[ \| \mb v \|^2_{\mc H^R} = \| v_1(|\cdot|) \|^2_{H^{m_d}(\B_R)} + \| v_2(|\cdot|) \|^2_{H^{m_d-1}(\B_R)},\]
cf.~\eqref{Def:SobRad}. If $R=1$, then we simply use the symbol $\mc H$, as before. In particular, $\| \cdot \|  = \| \cdot \|_{\mc H^1} $. Set
\[  \mb V(\mb v ,T)(\rho) := \left ( \begin{array}{c}  T^{\frac{2}{p-1}}  v_1(T\rho)  \\ T^{\frac{p+1}{p-1}}  v_2(T\rho) \end{array} \right ), \quad \mb \kappa(T) :=   \left ( \begin{array}{c}  (\tfrac{T}{T_0})^{\frac{2}{p-1}}c_p  \\ \  (\tfrac{T}{T_0})^{\frac{p+1}{p-1}} \frac{2}{p-1} c_p  \end{array} \right ),\]
and 
\[  \mb U(\mb v, T)  :=  \mb V(\mb v ,T) + \mb \kappa(T) - \mb \kappa(T_0) .\]

\begin{lemma}\label{Le:InitialData}
Let $\mb v \in \mc H^{T_0+\delta}$ for $\delta > 0$ sufficiently small. Then $T \mapsto \mb U(\mb v, T)$: $[T_0-\delta, T_0 + \delta ] \to \mc H$ is continuous. Furthermore, if $\| \mb v \|_{\mc H^{T_0 + \delta}} \leq \delta$ then
\[  \|\mb U(\mb v, T) \| \lesssim \delta \]
for all $T \in [T_0- \delta, T_0 + \delta]$.
\end{lemma}

\begin{proof}
For simplicity we prove the result only for $T_0 = 1$. The general case is analogous.  Let $\mb v \in \mc H^{1+\delta}$ for $0 < \delta \leq \frac{1}{2}$. To show continuity of the map $T \mapsto \mb U(\mb v, T)$ we consider the first component and estimate
\begin{align*}
& \|[\mb U(\mb v, T)]_1  - [\mb U(\mb v,  \widetilde T)]_1 \|_{H^{m_d}(\B)}   \\
 & =  \| T^{\frac{2}{p-1}} v_1(|T\cdot|)    - \widetilde T^{\frac{2}{p-1}} v_1( |\widetilde T\cdot|)  + T^{\frac{2}{p-1}} c_p  - \widetilde T^{\frac{2}{p-1} }c_p\|_{H^{m_d}(\B)}   \\
 & \lesssim   \| v_1(|T\cdot|) -v_1( |\widetilde T\cdot|) \|_{H^{m_d}(\B)}  + |T^{\frac{2}{p-1}}   -\widetilde T^{\frac{2}{p-1}}  | \|v_1(|T\cdot|) \|_{H^{m_d}(\B)}   \\
 & +  | T^{\frac{2}{p-1}} -\widetilde T^{\frac{2}{p-1}}  |
\end{align*}
for $T, \widetilde T \in [1-\delta, 1+ \delta]$. Scaling implies that for all $T \in [1-\delta, 1+\delta]$
\begin{align*}
\| v( |T\cdot|) \|_{H^{m_d}(\B)} \lesssim \| v(|\cdot|) \|_{H^{m_d}(\B_{1+\delta})}
\end{align*}
for $v \in H_r^{m_d}(\B_{1+\delta})$. Using this and the triangle inequality we infer that for all $v_1, \tilde v_1 \in H_r^{m_d}(\B_{1+\delta})$
\begin{align*}
  \| v_1 & (|T\cdot|) -v_1(| {\widetilde T} \cdot|) \|_{H^{m_d}(\B)}  \leq  \| v_1(|T \cdot|) -\tilde v_1( |T \cdot|) \|_{H^{m_d}(\B)}  \\
  &  + \| \tilde v_1(|T \cdot|) -\tilde v_1(|\widetilde T\cdot|) \|_{H^{m_d}(\B)} +  \| \tilde  v_1(| \widetilde T \cdot|) -v_1(|\widetilde T \cdot|) \|_{H^{m_d}(\B)}  \\
  & \lesssim \| v_1(|\cdot|) -\tilde v_1( |\cdot|) \|_{H^{m_d}(\B_{1+\delta})}   + \| \tilde v_1(|T \cdot|) -\tilde v_1(|\widetilde T \cdot|) \|_{H^{m_d}(\B)}.
\end{align*}
Since $C_e^{\infty}[0,1+\delta]$ is dense in $H_r^{m_d}(\B_{1+\delta})$ there is a $ \tilde v_1 \in C_e^{\infty}[0,1+\delta]$ such that 
$\| v_1(|\cdot|) -\tilde v_1( |\cdot|) \|_{H^{m_d}(\B_{1+\delta})}  < \varepsilon $ for given $\varepsilon > 0$. From the smoothness of $\tilde v_1$ we infer that
infer that
\begin{align*}
\lim_{T \to \tilde T} \| \tilde v_1(T \cdot)   -  \tilde v_1(\widetilde T \cdot)  \|_{H^{m_d}(\B)} = 0.
\end{align*}
Similar estimates can be obtained for the second component which yields the claimed continuity. For $\mb v  \in \mc H^{1+\delta}$, $\| \mb v \|_{\mc H^{1 + \delta}} \leq \delta$ and $T \in [1-\delta, 1 + \delta ]$,
\begin{align*}
\| [\mb U(\mb v, T) ]_1 \|_{H^{m_d}(\B)} &  \lesssim  T^{\frac{2}{p-1}} \| v_1(T \cdot)\|_{H^{m_d}(\B)} + c_p |T^{\frac{2}{p-1}}  -1|  \\\
& \lesssim\| v_1\|_{H^{m_d}(\B_{1+\delta})}+|T - 1|  \lesssim \delta.
\end{align*}
A similar estimate can be obtained for the second component and we infer that 
\[  \| \mb U(\mb v, T)  \| \lesssim \delta. \]
\end{proof}

\subsubsection{Variation of the blowup time}

\begin{theorem}\label{Th:GlobalExistence_Orig}
Set $R:=T_0+\frac{\delta}{c}$ for $c > 0$ sufficiently large and $\delta > 0$ sufficiently small. For every $\mb v \in \mc H^{R}$ with $\| \mb v \|_{\mc H^{R}} \leq \frac{\delta}{c^2}$ there exists a $T \in [T_0 - \frac{\delta}{c}, T_0 + \frac{\delta}{c}]$ and a function $\Phi  \in \mc X_{\delta}$ which satisfies
\begin{align}\label{Eq:Orig_Duhamel}
\Phi(\tau) = \mb S(\tau)\mb U(\mb v, T) + \int_0^{\tau} \mb S(\tau - \tau')  \mb N( \Phi (\tau')) d\tau'
\end{align}
for all $\tau > 0$.  In particular, $\Phi$ is the unique solution of this equation in $C([0,\infty),\mc H)$.
\end{theorem}

\begin{proof}
Let  $\mb v \in \mc H^{R}$. For $\delta$ and $c$ chosen appropriately, the smallness condition for $\mb v$ and Lemma \ref{Le:InitialData} imply that $\mb U(\mb v, T)$ satisfies the assumptions of Theorem \ref{Th:GlobalEx_ModEq} for all $T \in [T_0 -  \tfrac{\delta}{c},  T_0 +  \tfrac{\delta}{c}]$. Hence, for every such $T$ there exists a $\Phi_T := \mb \Phi (\mb U(\mb v,T))  \in \mc X_{\delta}$ satisfying
\begin{align*}
\Phi_T(\tau) = \mb S(\tau) & \mb U(\mb v, T)   + \int_0^{\tau} \mb S(\tau - \tau')  \mb N( \Phi_T (\tau')) d\tau'  \\
& - e^{\tau} \mb C(\Phi_T,\mb U(\mb v, T))
\end{align*}
for all $\tau > 0$. We show that there is a $T_{\mb v} \in  [T_0 -  \tfrac{\delta}{c},  T_0 +  \tfrac{\delta}{c}]$ such that 
$\mb C (\Phi_{T_{\mb v}}, \mb U(\mb v, T_{\mb v} )) = 0$. Note that $\mathrm{rg~}\mb C = \langle \mb g \rangle$, where $\mb g$ is the symmetry mode from Lemma \ref{Le:Spectrum}. Hence, it suffices to show that 
\begin{align*}
\big (\mb C (\Phi_{T_{\mb v}} , \mb U(\mb v, T_{\mb v} )) \big | \mb g \big ) = 0.
\end{align*}

We find that
\begin{align*}
 \left ( \int_0^{\infty} e^{-\tau} \mb P\mb N( \Phi_{T}(\tau))  d\tau \bigg | \mb g \right ) &  \lesssim  \| \mb g \|  \int_0^{\infty} e^{-\tau} \| \mb P\mb N( \Phi_{T}(\tau))\|  d\tau   \\
& \lesssim   \int_0^{\infty} e^{-\tau} \| \Phi_{T}(\tau)\|^2  d \tau  \lesssim \delta^2 .
\end{align*}
The key observation is that  
\[ \partial_T \mb{\kappa}(T)|_{T=T_0} =   \alpha \mb g\]
for a constant $\alpha > 0$.
By Taylor expansion of $\mb \kappa$ we get 
\[ \mb U(\mb v, T)  = \mb V(\mb v, T) + (T-T_0)  \alpha \mb g + (T-T_0)^2 \mb q_T   \]
for all $T \in [T_0 - \frac{\delta}{c}, T_0 + \frac{\delta}{c}] $ and for some remainder term $\mb q_T$.
Thus,
\begin{align*}
(\mb P \mb U(\mb v, T)| \mb g) & = (\mb P \mb V( \mb v, T)| \mb g)  +  \alpha (T-T_0) \| \mb g \|^2 +  O(\tfrac{\delta^2}{c^2})   \\
& =  \alpha (T-T_0) \| \mb g \|^2 +  O(\tfrac{\delta}{c^2})  
\end{align*}
by definition of $\mb V$ and the smallness of $\mb v$ in $\mc H^{T_0 + \frac{\delta}{c}}$. We notice that the order terms depend continuously on $T$. Summing up, we obtain that the equation
\begin{align*}
\big (\mb C (\Phi_{T} , \mb U(\mb v, T ))| \mb g \big ) =  \alpha (T-T_0) \| \mb g \|^2  +  O(\tfrac{\delta}{c^2}) +  O(\delta^2) = 0
\end{align*}
is equivalent to
\[ T= T_0 + F(T) \]
where $F(T) =  O(\tfrac{\delta}{c^2}) +  O(\delta^2)$.  For $c$ sufficiently large and $\delta = \delta(c)$ sufficiently small we get $|F(T)| \leq \frac{\delta}{c}$. Hence, the continuous function $T \mapsto T_0 + F(T) $ maps the interval $[T_0 - \frac{\delta}{c}, T_0 + \frac{\delta}{c}] $ to itself and has thus a fixed point at some $T = T_{\mb v}$. We therefore obtain a solution $\Phi_{T_{\mb v}} \in \mc X_{\delta} $ of the original equation \eqref{Eq:Orig_Duhamel}. For the uniqueness of the solution in $C([0,\infty),\mc H)$ we refer the reader to the proof of Theorem 4.11 in \cite{DonSch12}.
\end{proof}

\subsubsection{Proof of Theorem \ref{Th:Main}}

Choose $\delta, c > 0$ such that Theorem \ref{Th:GlobalExistence_Orig} holds and set $\delta' = \delta/c$. Let $(u_0,u_1) \in H^{\frac{d+1}{2}} \times H^{\frac{d-1}{2}}(\R^d)$ be radial functions, i.e.,  $(u_0,u_1) = (\tilde u_0(|\cdot|), \tilde u_1(|\cdot|))$,  that satisfy 
\begin{align*}
\|(u_0,u_1)-  u_{T_0}[0] \|_{H^{\frac{d+1}{2}} \times H^{\frac{d-1}{2}}(\B_{T_0+\delta'})}  \leq  \tfrac{\delta'}{c}.
\end{align*}
For $(f,g) = ( \tilde u_0, \tilde u_1  )-  u_{T_0}[0]$, cf.~Eq.~\eqref{Eq:Radial_NLW_Data}, this assumption implies that 
\begin{align*}
\| (f,g) \|_{\mc H^{T_0+\delta/c}}  \leq \tfrac{\delta}{c^2}.
\end{align*}
Hence, $\mb v :=  (f, g)$ satisfies the assumptions of Theorem \ref{Th:GlobalExistence_Orig}. We infer that there exists a $T \in [T_0 - \delta, T_0  + \delta]$ such that Eq.~\eqref{Eq:Orig_Duhamel} has unique solution $\Phi \in C([0,\infty), \mc H)$ with
\[ \| \Phi(\tau) \| \leq \delta e^{-\mu_p \tau}, \quad \forall \tau > 0. \]
Hence, $\Psi = \Phi + \mb c_p$ is a solution of Eq.~\eqref{Eq:First_order_css_nonl} (in the Duhamel sense) with initial data $\Psi(0) =\bs U(\bs v,T) + \mb c_p$. Consequently,  
\[ u(t,x) = (T-t)^{-\frac{2}{p-1}} \psi_1( - \log(T-t) + \log T , \tfrac{|x|}{T-t}) \]
is a radial solution of the original wave equation \eqref{Eq:NLW} with initial data 
\begin{align*}
u(0,x) =T^{-\frac{2}{p-1}} \psi_1( 0 , \tfrac{|x|}{T}) = u_{T_0}(0,x) + f(|x|) = u_0(x) \\
\partial_t u(0,x) =T^{-\frac{p+1}{p-1}} \psi_1( 0 , \tfrac{|x|}{T}) = \partial_t u_{T_0}(0,x)+ g(|x|) = u_1(x)
\end{align*}
for $x \in \B_{T}$. Furthermore, $u$ satisfies the estimates
\begin{align*}
(T- & t)^{\frac{2}{p-1} -\frac{d}{2}+k}  \|u(t,\cdot) - u_T(t,\cdot) \|_{\dot H^k(\B_{T-t})}  \\
& = (T-t)^{-\frac{d}{2}+k} \|\psi_1( - \log(T-t) + \log T , \tfrac{|\cdot|}{T-t}) - c_p\|_{\dot H^k(\B_{T-t})} \\
& = \|\psi_1( - \log(T-t) + \log T , |\cdot|) - c_p\|_{\dot H^k(\B)}   \\
& = \|\varphi_1( - \log(T-t) + \log T , |\cdot|)\|_{\dot H^k(\B)}  \\
& \lesssim  \|\varphi_1( - \log(T-t) + \log T , |\cdot|)\|_{H^{m_d}(\B)} \\
&  \lesssim \| \Phi( - \log(T-t) + \log T)\| \lesssim (T-t)^{\mu_p}.
\end{align*}
for $k = 0, \dots, \frac{d+1}{2}$. The bounds for the time derivative of the solution follow accordingly.

\section{Improvement of the topology - Proof of Theorem~\ref{Th:Main2}}

The verification of Theorem \ref{Th:Main2} is analogous to the proof of Theorem \ref{Th:Main} and we only discuss the main arguments. With the definitions of Section \ref{Sec:FuncSet} we introduce the product space $\mc {\tilde H} := H_r^{m_d-1} \times  H_r^{m_d-2}(\B)$ with norm
\[ \norm{\mb u }^2 = \| u_1(|\cdot|) \|^2_{H^{m_d-1}(\B)} +  \| u_2(|\cdot|) \|^2_{H^{m_d-2}(\B)}. \]

\subsubsection{Time evolution for the linearized problem}

We proceed as in Section \ref{Sec:LinWellPosed}. Since most proofs are similar or can even be copied verbatim we only sketch the main steps and point out differences. With the same notation as in Section \ref{Sec:FuncSet} we define
\begin{align*}
\norm{\mb u}^2_D := \|\D u_1\|^2_{\dot H^1(0,1)} + \|\D u_2\|^2_{L^2(0,1)}.
\end{align*}
\begin{lemma}\label{Le:Prop_Htilde}
We have 
\[ \norm{\mb u }^2  \simeq  \|\D u_1\|^2_{H^1(0,1)} + \|\D u_2\|^2_{L^2(0,1)} \simeq \norm{\mb u}^2_D\]  for all $\mb u \in \mc {\tilde H}$. Furthermore, $\D u_1 \in C[0,1]$ and $\D u_1(0) = 0$.
\end{lemma}

\begin{proof}
The equivalence of the parts involving only $u_1$ follows from Lemma \ref{Le:Aux_Norm} and the proof of Lemma \ref{Prop:EquivNorms}. For the second component the same methods can be used to show that
 \[\| u(|\cdot|) \|^2_{H^{m_d-2}(\B)} \simeq \sum_{n=0}^{m_d-2} \| (\cdot)^{n+1} u^{(n)} \|^2_{L^2(0,1)} \simeq  \|\D u\|^2_{L^2(0,1)}\]
for all $u \in H_r^{m_d-2}(\B)$. 
The properties of $\D u_1$ are a consequence of the density of $\Ceven$ in $\mc {\tilde H}$ and the embedding $H^1(0,1) \hookrightarrow L^{\infty}(0,1)$. 
\end{proof}

We note that Eq.~\eqref{Eq:Comm_Rel_LambdaD}  and  Eq.~\eqref{Eq:Comm_Rel_LaplaceD} of Lemma \ref{Le:Comm_RelD} hold for all $u \in H^{m_d-2} \cap C^{m_d-1}(0,1)$ and $u \in H^{m_d-1} \cap C^{m_d}(0,1)$, respectively. For the moment, we leave the value of $p$ unspecified and consider $\tilde{\mb L}_0$ as defined in Eq.~\eqref{Def:L0} on the domain
 \begin{align*}
& \mc  D(\tilde{\mb L}_0) :=  \left \{  \mb u \in  \mc {\tilde H} \cap C^{\infty}(0,1)^2: \D u_1 \in C^2[0,1], \right. \\
& \left. \phantom{!!!!!!!!!!!!!!!!!!!!!!!!!!!!!!!} \D u_2 \in C^1[0,1],  (\D u_2)(0) = 0   \right \}.
\end{align*}
The perturbation $\mb L'$ is defined as in Eq.~\eqref{Def:Perturbation} and Lemma \ref{Le:Compact_Pert} holds with $\mc H$ replaced by $\mc {\tilde H}$. It is easy to check that 
\begin{align*}
\mathrm{Re}  \big (  \D (\tilde{\mb L}_0  \mb u)_1&  \big | \D u_1 \big )_{\dot H^1(0,1)}  \\
&  + \mathrm{Re} \big (\D (\tilde{\mb L}_0  \mb u)_2 \big | \D u_2\big )_{L^2(0,1)}  \leq  (\tfrac{1}{2} - \tfrac{2}{p-1})\norm{ \mb u } ^2_{D}
\end{align*}
for all $\mb u \in  \mc  D(\tilde{\mb L}_0)$ and that $\mathrm{rg}(\mu - \tilde{\mb L}_0)$ is dense in $\mc {\tilde H}$ for $\mu = 1 - \tfrac{2}{p-1}$ . With $\mb L_0$ denoting the closure of $\tilde{\mb L}_0$, we use the same arguments as in Section \ref{Sec:LinWellPosed} to infer that $\mb L := \mb {L}_0 + \mb L'$, $\mc D(\mb{L}) = \mc D (\mb L_0 )$, generates a strongly-continuous semigroup $(\mb S(\tau))_{\tau > 0}$ of bounded operators on $\mc {\tilde H}$. As an analogue to Lemma \ref{Le:Spectrum} we obtain the following result for the spectrum of $\mb L$, where we fix $p=3$.

\begin{lemma}
Let $p=3$. If $\lambda \in \sigma(\mb L)$ and $\mathrm{Re} \lambda > -\frac{1}{2}$, then $\lambda = 1$. Furthermore, it is an eigenvalue and $\mathrm{ker}(1 - \mb L) = \langle \mb g \rangle$, where $\mb g = (1,2)$.
\end{lemma}

\begin{proof}
The assumptions on $\lambda$ imply that $\lambda \not \in \sigma(\mb L_0)$ and the fact that $\lambda$ is an eigenvalue follows from the compactness of $\mb L'$. By definition, if $\mb u \in  \mc {\tilde H}$, then $u_1$ and $u_2$ are $m_d-1$--times, respectively, $m_d-2$--times weakly differentiable. Sobolev embedding yields $u_1 \in C^{m_d-2}[\delta,1]$, $u_2 \in  C^{m_d-3}[\delta,1]$ for arbitrary $\delta > 0$. For $d \geq 7$, this already implies that eigenfunctions corresponding to the eigenvalue $\lambda$ satisfy the equation $(\lambda - \mb L) \mb u = 0$ in a classical sense on $(0,1)$. For $d=5$, one can use the definition of the closure to check that $u_1 \in H^{3}_{\mathrm{loc}}(0,1) \cap C^2(0,1)$, $u_2 \in H^{2}_{\mathrm{loc}}(0,1) \cap C^1(0,1)$ if $\mb u \in \mc D(\mb L)$. Hence, for all $d \geq 5$ odd, the first component of an eigenfunction solves Eq.~\eqref{Eq:Eigenvalue} in a classical sense on $(0,1)$. By smoothness of the coefficients on the open interval, we get that $u_1 \in C^{\infty}(0,1)$ and the application of $\D$ shows that $w := \D u_1$ solves Eq.~\eqref{Eq:EigenvalueD} on $(0,1)$. By Lemma \ref{Le:Prop_Htilde}, $w \in H^1(0,1) \cap C[0,1]$ and $w(0) = 0$.  We now argue as in the proof of Lemma \ref{Le:Spectrum}, cf.~also \cite{DonSch12}, to infer that $\lambda =1$ and that the corresponding eigenspace is spanned by $\mb g$.
\end{proof}

With similar arguments as in the proof of Lemma \ref{Le:LinearTimeEvol} we obtain the following result.
\begin{lemma}\label{Le:LinearTheory_Th2}
Let $p=3$. There exists a projection $\mb P \subset \mc B(\mc {\tilde H})$, $\mathrm{rg} \mb P = \langle \mb g \rangle$, that commutes with $\mb S(\tau)$ and 
\[ \mb S(\tau) \mb P \mb f = e^{\tau} \mb P \mb f  \]
for all $\mb f \in \mc {\tilde H}$ and all $\tau > 0$. Moreover,
\[ \norm{ (1 - \mb P) \mb S(\tau) \mb f }  \leq M e^{-(\frac12 - \varepsilon) \tau} \norm{ (1 - \mb P) \mb f  } \]
for all $\mb f \in \mc {\tilde H}$, $\tau > 0$, some $M \geq 1$ and some small  $\varepsilon >0$. 
\end{lemma}

\subsubsection{Lipschitz estimates for the nonlinearity}
In the following, $\mathcal B$ denotes the unit ball in $ \mc {\tilde H}$. For $p=3$, the nonlinear remainder is given by
\[ \mb N(\mb u)  := \left ( \begin{array}{c} 0 
\\ u_1^3 + 3 c_p u_1^2 \end{array} \right ). \] 

\begin{lemma}\label{Le:Nonlin_Th2}
The operator $\mb N$: $\mc {\tilde H}  \to \mc {\tilde H} $ satisfies
\begin{align*}
\norm{ \mb N(\mb u) - \mb N(\mb v) }  \lesssim( \norm{\mb u}  +\norm{ \mb v } )  \norm{ \mb u  -\mb v }
\end{align*}
for all $\mb u, \mb v \in \mathcal B \subset \mc {\tilde H}$.
\end{lemma}

\begin{proof}
For  $u\in H_r^{m_d-1}(\B)$, $m_d-1 = \frac{d-1}{2}$, we set $\hat u(\xi) := u(|\xi|)$, $\xi \in \R^d$. In the following we do not indicate the domain in the Sobolev norms, since it is always the unit ball $\B \subset \R^d$. The Sobolev embedding $W^{j+m, 2} \hookrightarrow W^{j,q}$ for $2 \leq q \leq \frac{2d}{d-2m}$ implies that
\[ \| \partial^{\alpha} \hat u \|_{L^{q}} \lesssim \|  \hat u \|_{H^{m_d-1} } \]
for $\alpha \in \N^d$, $0 \leq |\alpha| \leq \frac{d-3}{2}$ and $2 < q \leq \frac{2d}{1+2|\alpha|}$. We first consider the cubic part of the nonlinearity. To estimate the $L^2$--part we use H\"older's inequality with $q_1 = \frac{2d}{d-2}$, $q_2=2d$, $ \frac{1}{q_1} + \frac{2}{q_2}  = \frac{1}{2}$, to show that 
\begin{align*}
 \| \hat u^3  - \hat v^3  & \|_{L^2} = \| (\hat u - \hat v)( \hat u^2 + \hat v^2 + \hat u \hat v) \|_{L^2} \\ \lesssim 
&  \| \hat u - \hat v \|_{L^{q_1}}( \| \hat u\|^2_{L^{q_2}}  + \| \hat v\|^2_{L^{q_2}}   +
\| \hat u\|_{L^{q_2}} \| \hat v\|_{L^{q_2}} )  \\
& \lesssim ( \| \hat u\|_{H^{m_d-1} }^2 + \| \hat v\|_{H^{m_d-1} }^2 ) \|\hat u - \hat v\|_{H^{m_d-1} }.
\end{align*}
For higher order derivatives  we have to estimate terms of the form 
\[ \partial^{\beta}(\hat u - \hat v) \partial^{\alpha-\beta}( \hat u^2 + \hat v^2 + \hat u \hat v),\]
for $0 \leq \beta \leq \alpha$.
For $|\alpha| = \frac{d-3}{2}$, $\beta = \alpha$, we apply again H\"older's inequality and Sobolev embedding to get for example
\begin{align*}
\|(\partial^{\alpha}\hat u - \partial^{\alpha}\hat v)  \hat u^2\|_{L^2}  & \lesssim \|\partial^{\alpha}(\hat u - \hat v) \|_{L^{q_1}} \| \hat u\|^2_{L^{q_2}}    \lesssim\|\hat u - \hat v\|_{H^{m_d-1} } \| \hat u\|_{H^{m_d-1} }^2
\end{align*} 
for $q_1 = \frac{2d}{d-2}$, $q_2 = 2d$.  Since $\partial^{\alpha}\hat u^2$ is equal to a sum of terms of the form $ \partial^{\alpha_1} \hat u \partial^{\alpha_2} \hat u$, where $\alpha_1 + \alpha_2 = \alpha$, we infer that for $\beta =0$,
\begin{align*}
\|(\hat u - \hat v) \partial^{\alpha_1} \hat u  \partial^{\alpha_2} \hat u \|_{L^2}  & 
\lesssim \|\hat u - \hat v\|_{L^{q_1}} \|\partial^{\alpha_1} \hat u\|_{L^{q_2}} \|\partial^{\alpha_2} \hat u\|_{L^{q_3}} \\
&   \lesssim 
\|\hat u - \hat v\|_{H^{m_d-1} } \| \hat u\|_{H^{m_d-1} }^2,
\end{align*} 
where $q_1 = 2d$, $q_2 = \frac{2d}{1+2|\alpha_1|}$, $q_3= \frac{2d}{1+2|\alpha_2|}$, $\sum_{j=1}^3 \frac{1}{q_j} = \frac{1}{2}$. All other terms can be estimated similarly. 

For the quadratic part of the nonlinearity we set for example $q_1=2d, q_2 = \frac{2d}{d-1}$, to get
\begin{align*}
\| (\hat u - \hat v)( \hat u + \hat v) \|_{L^2}  & \lesssim  \| \hat u - \hat v \|_{L^{q_1}} \| \hat u + \hat v\|_{L^{q_2}}  \\
& \lesssim \|\hat u + \hat v\|_{H^{m_d-1} } \|\hat u - \hat v\|_{H^{m_d-1} }
\end{align*}
or, for $|\alpha| = \frac{d-3}{2}$, $q_1 =2d$,  $q_2 = \frac{2d}{d-1}$,
\begin{align*}
\| (\hat u - \hat v) \partial^{\alpha} \hat u \|_{L^2}  & \lesssim  \| \hat u - \hat v \|_{L^{q_1}} \| \partial^{\alpha}\hat u \|_{L^{q_2}}  \lesssim \|\hat u + \hat v\|_{H^{m_d-1} } \|\hat u\|_{H^{m_d-1} }.
\end{align*}
Estimates for the remaining terms follow from similar considerations. 
\end{proof}
With Lemma \ref{Le:LinearTheory_Th2} and Lemma \ref{Le:Nonlin_Th2}, Theorem \ref{Th:Main2} follows by proceeding as above, starting with Section \ref{Sec:NonlinearCP}. 

\appendix

\section{Hardy's inequality}

\begin{lemma}\label{Le:Hardy}
Let $\alpha \in \N$. Assume that $f \in C^{\infty}[0,1]$ satisfies $f^{(j)}(0) = 0$ for $j = 0, \dots ,\alpha-1$. 
Then,
\[ \|  (\cdot)^{-\alpha} f \|_{L^2} \lesssim \|  (\cdot)^{-\alpha+1} f' \|_{L^2}. \]
\end{lemma}
\begin{proof}
For $f = 0$, the assertion is trivial. Let $f \neq 0$. We use integration by parts, l'Hospital's rule and the Cauchy-Schwarz inequality to obtain the estimate
\begin{align*}
& \int_0^{1} \rho^{-2\alpha} |f(\rho)|^2 d \rho   \leq  \lim_{\rho\to 0} \left ( |2\alpha-1|^{-1} \rho^{-2\alpha+1} |f(\rho)|^2  \right)    \\
& + |2\alpha-1|^{-1} \int_0^{1}  \rho^{-2\alpha+1} ( f'(\rho) \overline{f(\rho)} + f(\rho) \overline{ f'(\rho)} )d\rho \\
&  \lesssim \int_0^{1} \rho^{-2\alpha+1}  \mathrm{Re}[{f'(\rho)\overline{f(\rho)}}] d\xi  \lesssim  \int_0^{1}
\rho^{-2\alpha+1}  |f'(\rho)||\overline{f(\rho)}| d\rho  \\
& \lesssim  \left( \int_0^{1}  \rho^{-2\alpha} |f(\rho)|^2 d\rho\right)^{1/2}  \left( \int_0^{1} \rho^{-2\alpha+2} |f'(\rho)|^2  d\rho \right)^{1/2}.
\end{align*}
This implies the claim.
\end{proof}

\section{Proof of Lemma \ref{Le:Aux_Norm}}\label{Proof:Lemma_Aux_Norm}

It suffices to show that the claimed inequality holds for all $\mb u \in \Ceven^2$. By density this can be extended to all of $\mc H$.  In the following we set
\begin{align*}
 \nabla_{\text{rad}}^n := \begin{cases} \Delta^{n/2}_{\rho}   & \text{ for } n \text{ even}, \\  
 \frac{d}{d\rho} \Delta^{(n-1)/2}_{\rho}     & \text{ for } n \text{ odd}, \end{cases} 
 \end{align*}
where $\Delta_{\rho}  u(\rho) = u''(\rho) + \tfrac{d-1}{\rho} u'(\rho)$. To abbreviate the notation we define
\begin{align}\label{Def:SigmajNorms}
\begin{split}
& \| u \|^2_{\Sigma_1} := \| u  \|^2_{L^2(0,1)} + \sum_{n=1}^{m_d} \| (\cdot)^{n-1} u^{(n)} \|^2_{L^2(0,1)},  \\
&\| u \|^2_{\Sigma_2} := \sum_{n=0}^{m_d-1} \| (\cdot)^{n} u^{(n)} \|^2_{L^2(0,1)}.
\end{split}
\end{align}
\begin{lemma}\label{Le:EqSigmaNorms}
Let $u \in \Ceven$. Then 
\begin{align*}
\| u \|_{\Sigma_1}  \lesssim \| u(|\cdot|) \|_{H^{m_d}(\B)} ,  \text{ and }  \quad \| u \|_{\Sigma_2} \lesssim \| u(|\cdot|) \|_{H^{m_d-1}(\B)}.
\end{align*}
\end{lemma}

\begin{proof}
We prove the first estimate. Let $ \nabla := (\partial_{1}, \dots, \partial_{d} )^T$. Then
\begin{align*}
 \sum_{n=0}^{m_d} \int_0^{1} \rho^{d-1} | \del^n u(\rho)|^2 d\rho  & \simeq \sum_{n=0}^{m_d} \int_{\B} | \nabla^n u(|\xi|)|^2 d\xi \\
 & \lesssim \sum_{|\alpha| \leq m_d} \|\partial^{\alpha} u(|\cdot|) \|^2_{L^2(\B)} =  \| u(|\cdot|) \|^2_{H^{m_d}(\B)}. 
\end{align*}
In view of this, it suffices to show that 
\begin{align}\label{Eq:Est3}
\begin{split}
\| u \|^2_{\Sigma_1}\lesssim \sum_{n=0}^{m_d} \int_0^{1} \rho^{d-1} | \del^n u(\rho)|^2 d\rho.
\end{split}
\end{align}
First, observe that 
\begin{align*}
& \| u \|^2_{L^2(0,1)} + \sum_{n=1}^{m_d} \| (\cdot)^{n-1} u^{(n)} \|^2_{L^2(0,1)} \\
&  \lesssim \| u\|^2_{L^2(0,1)}  +   \int_0^{1}  \rho^{d-1} |\del^{m_d} u(\rho)|^2 d\rho   + \sum_{n=1}^{m_d-1 } \int_0^{1}  \rho^{2(n-1)} |u^{(n)}(\rho)|^2 d\rho \\
& \lesssim  \int_0^{1} |u(\rho)|^2  +\int_0^{1}  \rho^{d-1} |\del^{m_d} u(\rho)|^2 d\rho  +  \sum_{n=1}^{m_d-1 } \int_0^{1} \rho^{2(n-1)} |\del^n u(\rho)|^2 d\rho.
\end{align*}
We show that 
\begin{align}\label{Eq:Est4}
\int_0^{1} \rho^{2(n-1)} |\del^n u(\rho)|^2 d\rho \lesssim \sum_{j=0}^{m_d} \int_0^{1} \rho^{d-1} | \del^j u(\rho)|^2 d\rho 
\end{align}
for all $n=1, \dots, m_d-1$. Recall that for radial functions the trace theorem, cf. for example \cite{Evans}, p.~258, implies that
\begin{align*}
  |\del^n u(1) |^2 \lesssim  \int_0^1 \rho^{d-1} |\del^n u(\rho)|^2 d\rho + \int_0^1 \rho^{d-1} |\del^{n+1} u(\rho)|^2 d\rho
 \end{align*}
for all $n = 0, \dots m_d-1$. Assume $m_d$ is odd and let $n = m_d-1$. Then integration by parts and the Cauchy inequality imply that
\begin{align*}
& \int_0^{1}  \rho^{d-3}    |\del^{m_d-1} u(\rho)|^2 d\rho    \\
& \lesssim \left |[\del^{m_d-1} u](1) \right |^2   + \int_0^{1}  \rho^{d-2} |\del^{m_d-1}  u(\rho)||\del^{m_d} u_1(\rho)|d\rho    \\
&  \lesssim  \left |[\del^{m_d-1} u](1) \right |^2   +  \frac{1}{\varepsilon}  \int_0^{1}  \rho^{d-1} |\del^{m_d}  u(\rho)|^2 d\rho   +  \varepsilon   \int_0^{1}  \rho^{d-3} |\del^{m_d-1}  u(\rho)|^2 d\rho,
\end{align*}
for any $\varepsilon > 0$.
For $m_d$ even one can easily check that the function $(\cdot)^{d-1} \del^{m_d-1} u$ satisfies the assumption of Lemma \ref{Le:Hardy} for $\alpha = (d+1)/2$. Hence, Hardy's inequality can be applied to obtain 
\begin{align*}
  \int_0^{1}  \rho^{d-3} &  |\del^{m_d-1}u(\rho)|^2 d\rho  =\int_0^{1}  \rho^{-d-1} | \rho^{d-1}  [ \del^{m_d-2} u]'(\rho)|^2 d\rho   \\
 & \lesssim \int_0^{1}  \rho^{-d+1} \left | \left  ( \rho^{d-1} [\del^{m_d-2}u ]'(\rho) \right )' \right |^2 d\rho 
 \lesssim \int_0^{1}  \rho^{d-1}  | \del^{m_d} u(\rho)|^2 d\rho. 
\end{align*}

Now these arguments can be iterated to get Eq.~\eqref{Eq:Est4}.  Finally,
\begin{align*}
\int_0^1 & |u(\rho)|^2 d\rho \lesssim |u(1)|^2 + \int_0^1 \rho |u(\rho)| |u'(\rho)| d\rho  \\
& \lesssim  |u(1)|^2  + \frac{1}{\varepsilon} \int_0^1 |u'(\rho)|^2 d\rho +   \varepsilon \ \int_0^1 |u(\rho)|^2 d\rho,
\end{align*}
for any $\varepsilon > 0$ and the first line of Eq.~\eqref{Eq:Est3} follows. The second estimate in Lemma \ref{Le:EqSigmaNorms} can be obtained analogously.
\end{proof}

\begin{lemma}\label{Le:Extension_Bounds}
Let  $u \in \Ceven$.  There exists a compactly supported function $U \in C^{m_d}[0,\infty)$ such that $U(\rho) = u(\rho)$ for all $\rho \in [0,1]$ and
\begin{align*}
 \| u(|\cdot|) \|_{H^{m_d}(\B)}  \lesssim \| U(|\cdot|)  \|_{H^{m_d}(\R^d)}  \lesssim \| u \|_{\Sigma_1}.
\end{align*}
Similarly, one can construct an extension $\tilde U \in C^{m_d-1}[0,\infty)$ such that $\tilde U(\rho) = u(\rho)$ on $[0,1]$ and
\begin{align*}
 \| u(|\cdot|) \|_{H^{m_d-1}(\B)}  \lesssim \| \tilde U(|\cdot|)  \|_{H^{m_d-1}(\R^d)}  \lesssim \| u \|_{\Sigma_2}.
\end{align*}
\end{lemma}

\begin{proof}
For $m \in \N$, let $f \in C^{\infty}[0,1]$ and let $\varphi \in C^{\infty}[0,\infty)$ be a monotonically decreasing function such that $\varphi = 1$ on $[0,\frac{5}{4}]$ and $\varphi = 0$ on $[\frac{3}{2}, \infty)$. We define
\begin{equation*}
\mc E_m f(\rho) :=  \begin{cases} f(\rho) &\mbox{for } \rho \in [0,1] \\ 
\varphi(\rho) h_m(\rho) & \mbox{for } \rho \in (1,3/2) \\
0 & \mbox{for } \rho \in  [3/2,\infty) \end{cases} 
\end{equation*}
where
\begin{align*}
\begin{split}
 h_m(\rho) :=  \begin{dcases}  -f(2-\rho) +  \sum_{n = 0}^{(m-1)/2} \frac{2(\rho-1)^{2n}  }{(2n)!}  f^{(2n)}(1)  & m \text { is odd,} \\ 
f(2-\rho) +  \sum_{n = 1}^{m/2} \frac{2(\rho-1)^{2n-1}  }{(2n-1)!}  f ^{(2n-1)}(1) & m \text { is even.}  \end{dcases}. 
\end{split}
\end{align*}
Then, $\mc E_m f \in C^{m}[0,\infty)$ and $\mc E_m f|_{[0,1]} = f$. We define $U := \mc E_{m_d} u$.
Our aim is to prove the estimate
\begin{align}\label{Eq:Est1}
\begin{split}
\|U(|\cdot|)\|^2_{L^2(\R^d)} + \|\nabla^{m_d} U(|\cdot|) \|^2_{L^2(\R^d)}  &  \lesssim \| u\|^2_{\Sigma_1}.
\end{split}
\end{align}
Given that Eq.~\eqref{Eq:Est1} holds,  we can use the fact that $H^{m}(\R^d)$ can be equivalently defined in terms of the Fourier transform to infer that
\begin{align*}
\| u(|\cdot|) \|^2_{H^{m_d}(\B)}  & \leq  \| U(|\cdot|) \|^2_{H^{m_d}(\R^d)}  \lesssim  \| \langle \cdot  \rangle ^{m_d} \mc F[ U(|\cdot|) ] \|^2_{L^2(\R^d)} \\
 & \lesssim  \|U(|\cdot|) \|^2_{L^2(\R^d)} + \|\nabla^{m_d} U(|\cdot|) \|^2_{L^2(\R^d)}  \lesssim \|  u\|^2_{\Sigma_1}.
\end{align*}
If $m_d$ is odd we estimate
\begin{align*}
\begin{split}
& \|U(|\cdot|)   \|^2_{L^2(\R^d)}   \lesssim \int_{0}^{\infty} \rho^{d-1} |U(\rho)|^2 d\rho \\
& =  \int_{0}^{1} \rho^{d-1} |u(\rho)|^2 d\rho  + \int_{1}^{\frac{3}{2}} \rho^{d-1}  |\varphi(\rho) h_{m_d}(\rho)|^2 d\rho   \\
&  \lesssim  \| u \|^2_{L^2(0,1)} + \| \varphi \|^2_{L^{\infty}[0,\infty)} \bigg ( \int_1^{\frac{3}{2}} \rho^{d-1} |u(2-\rho)|^2 d\rho +   \sum_{n = 0}^{(m_d-1)/2} |u^{(2n)}(1)|^2 \bigg ) \\
&  \lesssim  \| u \|^2_{L^2(0,1)} + \int_{\frac12}^{1}  |u(\rho)|^2 d\rho   + \sum_{n = 0}^{(m_d-1)/2} |u^{(2n)}(1)|^2  \lesssim 
\| u \|^2_{\Sigma_1},
\end{split}
\end{align*}
where we used the fact that 
\begin{align*}
|u^{(n)} (1)|^2 \lesssim  \|u \|^2_{H^{m_d}(\delta, 1)}  \lesssim \| u\|^2_{L^2(0,1)} + \sum_{j=1}^{m_d} \| (\cdot)^{j-1} u^{(j)} \|^2_{L^2(0,1)}  
\end{align*}
for $n = 0, \dots, m_d - 1$.  For the derivative we get
 \begin{align*}  \int_{0}^{\infty} \rho^{d-1} |\del^{m_d}  U_1(\rho)|^2 d\rho  =  & \int_0^1 \rho^{d-1} |\del^{m_d}  u(\rho)|^2 d\rho   \\
 & + \int_1^{\frac32} \rho^{d-1}  \left |\del ^{m_d} \left [\varphi(\rho) h_{m_d}(\rho) \right] \right |^2 d\rho.
\end{align*}
To bound the first integral we exploit the fact that for  $d \geq 5$ odd and $m \in \N$ there exist constants $c_j^{(d,m)} \in \R$ such that
\begin{align}\label{Eq:Nablak}
\del^{m} u(\rho) = \sum_{n=1}^{m} c_n^{(d,m)} \rho^{n-m} u^{(n)}(\rho).
\end{align}
Hence,
\begin{align*}
\int_0^1 \rho^{d-1} |\del ^{m_d}  u(\rho)|^2 d\rho \lesssim \sum_{n=1}^{m_d} \int_0^1 |\rho^{n-1} u^{(n)}(\rho)|^2 d\rho \lesssim \| u \|^2_{\Sigma_1}.
\end{align*}
It remains to estimate the second term. It follows from Eq.~\eqref{Eq:Nablak} and the Leibniz rule that there exist constants $c_{n,j} ^{(d)} \in \R$ such that
\begin{align*}
\nabla_{\rho} ^{m_d} \left [\varphi(\rho) h_{m_d}(\rho) \right]  = \sum_{n=1}^{m_d}  \sum_{j=0}^{n} c_{n,j}^{(d)} \rho^{n-m_d} \varphi^{(n-j)}(\rho) h_{m_d}^{(j)} (\rho).
\end{align*}
Hence,
\begin{align*}
 \int_1^{\frac32} \rho^{d-1}  &  \left |\del ^{m_d} \left [\varphi(\rho) h_{m_d}(\rho) \right] \right |^2 d\rho     \\
 &  \lesssim \sum_{n=1}^{m_d}  \sum_{j=0}^{n} \int_1^{\frac32} \rho^{2n-2}  \left  |\varphi^{(n-j)}(\rho) h_{m_d}^{(j)} (\rho)\right |^2 d\rho  \\
 & \lesssim  \left ( \sum_{n=0}^{m_d} \|\varphi^{(n)}\|^2_{L^{\infty}[0,\infty)}  \right) \left( \sum_{n=0}^{m_d} \int_1^{\frac32} \rho^{d-1}  \left  | h_{m_d}^{(n)} (\rho)\right |^2 d\rho  \right )  \\
 & \lesssim   \sum_{n=0}^{m_d}  \int_1^{\frac32}  | u^{(n)}(2-\rho)|^2  d\rho +  \sum_{n=0}^{(m_d-1)/2}  \ |u^{(2n)}(1)|^2 \\
  & \lesssim   \sum_{n=0}^{m_d}  \int_{\frac12}^{1}  | u^{(n)}(\rho)|^2  d\rho + \|  u \|^2_{\Sigma_1} \lesssim \| u \|^2_{\Sigma_1}.
\end{align*}
If $m_d$ is even the proof works similarly. In fact, the extension was constructed in such a way, that the boundary terms involve only derivatives that can be bounded by the $\Sigma$--norm. The proof for the second estimate is analogous.
\end{proof}


\section{Proof of Lemma \ref{Prop:EquivNorms}}\label{App:Equivalent_norms}

Again, it suffices to prove the inequality for all  $\mb u \in \Ceven^2$.
We split the proof into several lemmas and use the result of Lemma \ref{Le:Aux_Norm}.
With the definition \eqref{Def:SigmajNorms} we set
\[ \| \mb u \|^2_{\Sigma} := \| u_1 \|^2_{\Sigma_1} +\| u_2 \|^2_{\Sigma_2} \] 
for $\mb u = (u_1,u_2)  \in \Ceven^2$.

\begin{lemma}
We have that $\| \mb u \|_{D} \lesssim \| \mb u \|_{\Sigma}$,
for all $\mb u \in \Ceven^2$.
\end{lemma}

\begin{proof}
With Eq.~\eqref{Eq:DdSum} and the triangle inequality we immediately obtain 
\begin{align*}
\begin{split}
\| \D u_1 \|_{\dot H^2(0,1)}^2 \lesssim \sum_{n=1}^{m_d} \| (\cdot)^{n-1} u_1^{(n)} \|^2_{L^2(0,1)}, \\
\| \D u_2 \|_{\dot H^1(0,1)}^2 \lesssim \sum_{n=0}^{m_d-1} \| (\cdot)^{n} u_2^{(n)} \|^2_{L^2(0,1)}.
\end{split}
\end{align*}
We use the Sobolev embedding $H^1(0,1) \hookrightarrow L^{\infty}(0,1)$ and the fact that $\D u_2(0) = 0$ to infer that 
\begin{align*} 
\begin{split}
\big | [\D u_1]'(1) + \D u_2(1) \big |^2  & \lesssim \big  | [\D u_1]'(1) \big |^2 + \big |\D u_2(1) \big |^2  \\
&  \lesssim \| [\D u_1]' \|^2_{H^1(0,1)} + \| \D u_2\|^2_{\dot H^1(0,1)}.
\end{split}
\end{align*}
Now, 
\begin{align*}
\begin{split}
\| [\D u_1]' \|^2_{L^2(0,1)} \lesssim & \sum_{n=0}^{m_d-1} \| (\cdot)^{n} u_1^{(n)} \|^2_{L^2(0,1)} \\
&  \lesssim  \| u_1 \|^2_{L^2(0,1)} +  \sum_{n=1}^{m_d} \| (\cdot)^{n-1} u_1^{(n)} \|^2_{L^2(0,1)},  
\end{split}
\end{align*}
which implies the claim.
\end{proof}

\begin{lemma} For all  $\mb u \in \Ceven^2$ the following inequality holds
\[\| \D u_1 \|^2_{H^2(0,1)} +  \| \D u_2 \|^2_{H^1(0,1)}   \lesssim   \| {\mb u} \|^2_{D}.\]
\end{lemma}

\begin{proof}
Set $w_j:=\D u_j$ for $j=1,2$. Since $w_j(0)=0$ for $j=1,2$, the above inequality is true if  
\begin{align*}
\|w_1'\|^2_{L^2(0,1)} \lesssim  |w_1'(1) + w_2(1)|^2 +\|w_1''\|^2_{L^2(0,1)} + \|w_2'\|^2_{L^2(0,1)}.
\end{align*}
By the fundamental theorem of calculus $\int_{\rho}^1 w_1''(s)ds = w_1'(1) - w_1'(\rho)$ and $\int_{\rho}^1 w_2'(s)ds = w_2(1) - w_2(\rho).$
Hence,
\[ |w'_1(\rho) + w_2(\rho)| \leq  |w'_1(1) + w_2(1)| + \|w_1''\|_{L^2(0,1)}  + \|w_2'\|_{L^2(0,1)}. \]
Using this together with Sobolev embedding we obtain
\begin{align*}
|w'_1(\rho)| & \leq |w'_1(\rho) + w_2(\rho)| + |w_2(\rho)|  \\
&  \lesssim|w'_1(1) + w_2(1)| + \|w_1''\|_{L^2}  + \|w_2'\|_{L^2(0,1)}.
\end{align*}
Squaring and integrating implies the claim.
\end{proof}

Lemma \ref{Prop:EquivNorms} follows from Lemma \ref{Le:Aux_Norm} in combination with the following result.
\begin{lemma} We have that 
\[ \| \mb u \|^2_{\Sigma}  \lesssim  \| \D u_1 \|^2_{H^2(0,1)} +   \| \D u_2 \|^2_{H^1(0,1)}, \]
for all $\mb u  \in \Ceven^2$.
\end{lemma}

\begin{proof}
We show that 
\begin{align}\label{Eq:Est3_1}
 \| (\cdot)^{n} u^{(n)} \|_{L^2(0,1)} \lesssim \| \D u \|_{H^1(0,1)} 
\end{align}
for $n = 0, \dots, m_d-1$ and all $u \in \Ceven$, and 
\begin{align}\label{Eq:Est3_2}
\| (\cdot)^{n-1} u^{(n)} \|_{L^2(0,1)} \lesssim C \| \D u \|_{H^2(0,1)} 
\end{align}
for $ n = 1, \dots, m_d$, by using the fact that $\K \D = I$ on $\Ceven$.  Let $u \in \Ceven$ and set $w := \D u$. Then, $w \in C^{\infty}[0,1]$, $w^{(2n)}(0) =0$, $n \in \N_0$, and  $\K w = u$. 
One can easily check that there exist constants $\alpha_{n,j}, \tilde \alpha_{n,j}  \in \R$ such that 
\begin{align*}
\rho^{n} (\K w)^{(n)}(\rho)  = \sum_{j = 0}^{n} \alpha_{n,j} K_{d-2j}w = \sum_{j=(d-3)/2-n}^{(d-3)/2} \tilde \alpha_{n,j} \rho^{-2j-1} \mc K^{j}w(\rho) 
\end{align*}
for $n= 0, \dots, m_d-2 = (d-3)/2$.
Furthermore,
\[  \rho^{m_d-1} (\K w)^{(m_d-1)}(\rho) =  \sum_{j=0}^{(d-3)/2}  \alpha_j \rho^{-2j-1} \mc K^{j}w(\rho)   +w'(\rho). \]
Since $w(0)=0$ and 
\[ \left[\tfrac{d^{k}}{d \rho^{k}} \mc K^{n}w \right](0) =0\]
for $n \in \N_0$, $k= 0,\dots, 2n$, repeated application of Hardy's inequality yields
\begin{align*}
\begin{split}
\int_0^{1} \rho^{-4n-2} | \mc K^n w(\rho)|^2 d\rho \leq C_n \int_0^{1} |w'(\rho)|^2 d\rho
\end{split}
\end{align*}
for $n \in \N_0$ and some constante $C_n >0$. This in particular implies that 
\[ \int_0^1 |\rho^{n} (\K w)^{(n)} (\rho)|^2 d\rho  \lesssim \int_0^{1} |w'(\rho)|^2 d\rho  \] 
which proves Eq.~\eqref{Eq:Est3_1}. By the fundamental theorem of calculus
\[ w(\rho) =  \rho w'(0) + \int_0^{\rho} \int_0^{s} w''(t) dt ds. \]
Upon setting $\mc Vw(\rho):= \int_0^{\rho} w(s)ds$  we infer that
\begin{align*}
\K w(\rho)  = \K \mc V^2w''(\rho) + k_d w'(0)
\end{align*}
for some constant $k_d > 0$.  Using this we obtain for $n= 1, \dots, m_d-2$ and  constants $\beta_n, \tilde \beta_n \in \R$,
\begin{align*}
\rho^{n-1}  (\K w)^{(n)}(\rho)   & = \sum_{j = 0}^{n} \beta_{n,j} \rho^{-1} K_{d-2j} \mc V^2w''(\rho)  \\
& = \sum_{j=(d-3)/2-n}^{(d-3)/2} \tilde \beta_{n,j} \rho^{-2j-2} \mc K^{j} \mc V^2w''(\rho) .
\end{align*}
 Furthermore,  
\begin{align*}
\rho^{m_d-2} (\K w)^{(m_d-1)} (\rho) = \sum_{j=0}^{(d-3)/2} \gamma_{j} \rho^{-2j-2} \mc K^{j} \mc V^2w''(\rho) + \rho^{-1} \mc V w''(\rho)
\end{align*}
and
\begin{align*}
\rho^{m_d-1} (\K w)^{(m_d)} (\rho) = \sum_{j=0}^{(d-3)/2} &  \tilde \gamma_{j}\rho^{-2j-2} \mc K^{j} \mc V^2w''(\rho) \\ 
&  + \gamma_d \rho^{-1} \mc V w''(\rho) + w''(\rho)
\end{align*}
for $ \tilde \gamma_{j}, \gamma_{j}, \gamma_d \in \R$. By repeated application of Hardy's inequality we can now show that 
\begin{align*}
\begin{split}
\int_0^{1} \rho^{-4n-4} | \mc K^{n} \mc V^2w''(\rho) )|^2 d\rho \leq C_n \int_0^{1} |w''(\rho)|^2 d\rho
\end{split}
\end{align*}
for $n \in \N$ and  a constant $C_n >0$. This implies Eq.~\eqref{Eq:Est3_2}.
\end{proof}

\pagestyle{plain}
\bibliography{references}
\bibliographystyle{plain}

\end{document}